\documentclass[12pt]{amsart}
\usepackage{a4wide}
\usepackage{amsfonts}
\usepackage{amsmath}
\usepackage{amssymb}
\usepackage{amsthm}
\usepackage{bm} 

\newtheorem{lem}{Lemma}[section]
\newtheorem{prop}[lem]{Proposition}
\newtheorem{thm}[lem]{Theorem}

\newtheorem{cor}[lem]{Corollary}

\newtheorem*{question1}{Question 1}
\newtheorem*{question2}{Question 2}
\newtheorem*{question3}{Question 3}
\newtheorem*{conjecture}{Conjecture}

\def\phi{\varphi}

\theoremstyle{definition}
\newtheorem{rem}[lem]{Remark}
\newtheorem{example}[lem]{Example}

\def\mysection#1{\section{#1}\setcounter{equation}{0}}

\def\CCC{{\mathbb C}}

\def\TT{{\mathcal T}}
\def\NNN{{\mathbb N}}
\def\RRR{{\mathbb R}}
\def\DDD{{\mathbb D}}
\def\TTT{{\mathbb T}}
\def\ZZZ{{\mathbb Z}}

\def\UU{{\mathcal U}}

\def\FF{\mathcal F}

\def\Th{\Theta}
\def\la{\lambda}
\def\<{\langle}
\def\>{\rangle}

\def\Im{\mathop{\rm Im}\nolimits}

\def\Re{\mathop{\rm Re}\nolimits}
\def\Im{\mathop{\rm Im}\nolimits}

\def\clos{\mathop{\rm clos}\nolimits}
\def\supp{\mathop{\rm supp}\nolimits}

\def\htl{h^\Theta_{\lambda}}
\def\ktl{k^\Theta_{\lambda}}

\def\httl{\tilde h^\Theta_{\lambda}}
\def\kttl{\tilde k^\Theta_{\lambda}}

\def\om{\omega}
\def\Id{\mathrm{Id}}

\newcommand{\Ath}{A^\Theta}
\newcommand{\Kth}{K_\Theta}
\newcommand\Rea{\mathop{\rm Re}\nolimits}

\title[Truncated Toeplitz operators]{Bounded symbols and Reproducing Kernel 
Thesis for truncated Toeplitz operators}
\author[A. Baranov]{Anton Baranov}
\author[I. Chalendar]{Isabelle Chalendar}
\author[E. Fricain]{Emmanuel Fricain}
\author[J. Mashreghi]{Javad Mashreghi}
\author[D. Timotin]{Dan Timotin}

\address[A. Baranov]{Department of Mathematics and Mechanics, St. Petersburg State
University, 28, Universitetskii pr., St. Petersburg, 198504, Russia}
\email{a.baranov@ev13934.spb.edu}
\address[I. Chalendar and E. Fricain]{Universit\'e de Lyon;
Universit\'e Lyon 1; INSA de Lyon; Ecole Centrale de Lyon;
CNRS, UMR5208, Institut Camille Jordan; 43 bld. du 11 novembre 1918,
F-69622 Villeurbanne Cedex, France}
\email{chalenda@math.univ-lyon1.fr, fricain@math.univ-lyon1.fr}
\address[J. Mashreghi]{D\'epartement de Math\'ematiques et de Statistique,
         Universit\'e Laval,
         Qu\'ebec, QC,
         Canada G1K 7P4.}
\email{javad.mashreghi@mat.ulaval.ca}

\address[D. Timotin]{Institute of Mathematics of the 
Romanian Academy, PO Box 1-764, Bucharest 014700, Romania}
\email{Dan.Timotin@imar.ro}

\thanks{This work was partially supported by funds from NSERC (Canada),
Centre International de Rencontres Mathématiques (France) and RFBR (Russia). }

\begin{document}
\begin{abstract}
Compressions of Toeplitz operators to coinvariant subspaces of $H^2$ 
are called \emph{truncated Toeplitz operators}. We study two questions 
related to these operators. The first, raised by Sarason, is whether 
boundedness of the operator implies the existence of a bounded symbol; 
the second is the Reproducing Kernel Thesis. We show that in general 
the answer to the first question is negative, and we exhibit some 
classes of spaces for which the answers to both questions are positive.
\end{abstract}

\renewcommand{\baselinestretch}{1.3}
\normalsize
\bibliographystyle{acm}
\maketitle
\mysection{Introduction}\label{introduction}

Truncated Toeplitz operators on model spaces have been formally 
introduced by Sarason in~\cite{Sarason}, although special cases 
have long ago appeared in literature, most notably as  model 
operators for contractions with defect numbers one and for their 
commutant. They are naturally related to the classical 
Toeplitz and Hankel operators on the Hardy space. This is a new 
area of study, and it is remarkable that many simple questions 
remain still unsolved. As a basic reference for their main 
properties,~\cite{Sarason} is invaluable; further study can be 
found in~\cite{Cima-2009,Cima-2008,Ross-Garcia} 
and in~\cite[Section 7]{Sarason1}.

The truncated Toeplitz operators live on the model spaces $\Kth$. These are subspaces of $H^2$ (see Section~\ref{se:preliminaries} 
for precise definitions) that have attracted attention in the last decades; they are relevant in
various subjects such as for instance spectral theory for general linear operators \cite{Nik86}, control theory \cite{Nikolski-controle}, and Nevanlinna domains connected to rational approximation \cite{Fedorovskii}. 
Given a model space $\Kth$ and a function $\phi\in L^2$, the truncated 
Toeplitz operator $\Ath_\phi$ is  defined on a dense subspace of 
$\Kth$ as the compression to $\Kth$ of multiplication by $\phi$. 
The function $\phi$ is then called a symbol of the operator, and 
it is never uniquely defined.

In the particular case where $\phi\in L^\infty$ the operator $\Ath_\phi$ is bounded. 
In view of well-known facts about classical Toeplitz and Hankel 
operators, it is natural to ask whether the converse is true, 
that is, if a bounded truncated Toeplitz operator has necessarily 
a bounded symbol. This question has been posed in~\cite{Sarason}, 
where it is noticed that it is nontrivial even for rank one operators. 
In the present paper we will provide a class of inner functions $\Theta$ 
for which there exist rank one truncated Toeplitz operators 
on $\Kth$ without bounded symbols. On the other hand, we obtain positive results 
for some basic examples of model spaces. Therefore  
the situation  is quite different 
from the classical Toeplitz and Hankel operators.

The other natural question that we address is the Reproducing Kernel 
Thesis for truncated Toeplitz operators. Recall that an operator  on a  
reproducing kernel Hilbert space is said to satisfy the \emph{Reproducing 
Kernel Thesis} (RKT) if
its boundedness is determined  by its behaviour on the reproducing
kernels. This property has  been studied for several 
classes of operators:  Hankel and Toeplitz operators on
the Hardy space of the unit disc ~\cite{Bon,Holland-Walsh88,Smith05},  Toeplitz
operators on the Paley--Wiener space~\cite{martin},   semicommutators of Toeplitz operators~\cite{Nik86},
Hankel operators on the Bergman space~\cite{axler,Harper-Smith}, and Hankel operators on the Hardy space of the bidisk~\cite{ferguson,pott-sadoski}. 
It appears thus natural to ask the corresponding question 
for truncated Toeplitz operators. We will show that in this 
case it is more appropriate to assume the boundedness of the operator 
on the reproducing kernels as well as on a related ``dual'' family, 
and will discuss further its validity for certain model spaces.

The paper is organized as follows.
The next two sections contain preliminary 
material concerning model spaces and truncated 
Toeplitz operators. Section~4 introduces the main two problems we 
are concerned with: existence of  bounded symbols and the
Reproducing Kernel Thesis. The counterexamples are presented 
in Section~5; in particular, Sarason's  question on the general 
existence of bounded symbols is answered in the negative. Section~6 
exhibits some classes of model spaces for which the answers to both 
questions are positive. Finally, in Section~7 we present another class 
of well behaved truncated Toeplitz operators, namely operators with positive symbols.

\mysection{Preliminaries on model spaces}\label{se:preliminaries}

Basic references for the content of 
this section are~\cite{duren,garnett} 
for general facts about Hardy spaces and~\cite{Nik86} for model spaces and operators.

\subsection{Hardy spaces}\label{sse:hardy}

The Hardy space $H^p$ of the unit disk $\DDD=\{z\in\CCC:\, |z|<1\}$ is the space of analytic functions $f$ on $\DDD$ satisfying $\|f\|_p<+\infty$, where 
\[
\|f\|_p=\sup_{0\leq r<1}\left(\int_0^{2\pi}|f(re^{i\theta})|^p\frac{d\theta}{2\pi}\right)^{1/p},
\qquad 1\leq p<+\infty.
\]
The algebra of bounded analytic functions on $\DDD$ is denoted by $H^\infty$. We denote also $H_0^p=zH^p$. 
Alternatively, $H^p$ can be identified (via radial limits) with the subspace 
of functions $f\in L^p=L^p(\TTT)$ for which $\hat f(n)=0$ for all $n<0$. Here $\TTT$ denotes the unit circle with normalized Lebesgue measure $m$.

For any $\phi\in L^\infty$, we denote by  $M_\phi f=\phi f$ 
the \emph{multiplication operator} on $L^2$; we have $\|M_\phi\|=\|\phi\|_\infty$. 
The \emph{Toeplitz} and \emph{Hankel} operators on $H^2$ are given by the formulas
\[
\begin{split}
T_\phi&=P_+M_\phi,\qquad T_\phi:H^2\to H^2;\\
H_\phi&=P_-M_\phi,\qquad H_\phi:H^2\to H^2_-,
\end{split}
\]
where $P_+$ is the Riesz projection from $L^2$ onto $H^2$ and $P_-=I-P_+$ is the orthogonal projection from $L^2$ onto $H^2_-=L^2\ominus H^2$.
In case where $\phi$ is analytic, $T_\phi$ 
is just the restriction of $M_\phi$ to $H^2$. We have $T_\phi^*=T_{\bar\phi}$ 
and $H_\phi^*=P_+M_{\bar\phi}P_-$;  we also denote  $S=T_z$ the usual shift operator on $H^2$. 

Evaluations at points $\lambda\in\DDD$ are bounded functionals on $H^2$ and the corresponding reproducing kernel is
$k_\lambda(z)=(1-\bar\lambda z)^{-1}$; thus, $f(\lambda)=\langle f,
k_\lambda\rangle$, for every function $f$ in $H^2$. If $\phi\in  H^\infty$, then $k_\lambda$ is an
eigenvector for $T_\phi^*$, and $T_\phi^*
k_\lambda=\overline{\phi(\lambda)}k_\lambda$.
By normalizing $k_\lambda$ we obtain
$h_\lambda=\frac{k_\lambda}{\|k_\lambda\|_2}=\sqrt{1-|\lambda|^2}k_\lambda$.

\subsection{Model spaces}\label{sse:model}

Suppose now $\Theta$ is an inner function, that is, a function in $H^\infty$ 
whose radial limits are of modulus one almost everywhere on $\TTT$. 
In what follows we consider only nonconstant inner functions.
We define the corresponding {\emph {shift-coinvariant subspace}} generated by $\Theta$ 
(also called \emph{model space}) by the formula $K_\Theta^p=H^p\cap \Theta \overline{H_0^p}$, $1\leq p<+\infty$. 
We will be especially interested in the Hilbert case, that is, when $p=2$. In this 
case we write $K_\Theta=K_\Theta^2$;  it is easy to see that $K_\Theta$ is also given by 
\[
K_\Theta=H^2\ominus\Theta H^2=\left\{f\in H^2:\langle f,\Theta g\rangle=0,\forall g\in H^2\right\}.
\]
The orthogonal projection of $L^2$ onto $K_\Theta$ is denoted by 
$P_\Theta$; we have $P_\Theta=P_+-\Theta P_+\bar\Theta$. 
Since the Riesz projection 
$P_+$ acts boundedly on $L^p$, $1<p< \infty$, 
this formula shows that $P_\Theta$ can also be regarded as a bounded operator from $L^p$  onto 
$K_\Theta^p$, $1<p<\infty$.

The reproducing kernel in $K_\Theta$ for a 
point $\lambda\in\DDD$ is the function
\begin{equation}\label{eq:klth}
k_\lambda^\Theta(z)=(P_\Th k_\la)(z)=
\frac{1-\overline{\Theta(\lambda)}\Theta(z)}{1-\bar\lambda
z};
\end{equation}
we denote by $h_\lambda^\Theta$ 
the normalized reproducing kernel, 
\begin{equation}\label{eq:cafe}
h_\lambda^\Theta(z)=\sqrt{\frac{1-|\lambda|^2}{1-|\Theta(\lambda)|^2}}\,
k_\lambda^\Theta(z).
\end{equation}
Note that, according to~(\ref{eq:klth}), we have the orthogonal decomposition
\begin{equation}\label{eq:orth}
k_\lambda=k_\lambda^\Theta + \Theta \overline{\Theta(\lambda)} k_\lambda.
\end{equation}

We will use the antilinear isometry $J:L^2\to L^2$, given by 
$J(f)(\zeta)=\overline{\zeta f(\zeta)}$; it maps $H^2$ 
into $\overline{H^2_0} =L^2\ominus H^2=H^2_-$ and conversely.
More often another antilinear isometry $\om=\Th J$ will appear, 
whose main properties are summarized below.

\begin{lem}\label{le:omega}
Define, for $f\in L^2$, $\om(f)(\zeta)=\overline{\zeta f(\zeta)}\Th(\zeta)$. Then:

{\rm (i)} $\om$ is antilinear, isometric, onto; 

{\rm (ii)} $\om^2=\Id$;

{\rm (iii)} $\om P_\Th=P_\Th\om$ (and therefore $K_\Th$ reduces $\om$), $\om(\Th H^2)=H^2_-$ and $\om(H^2_-)=\Th H^2$.
\end{lem}

We define the \emph{difference quotient} $\kttl=\om(\ktl)$ and $\httl=\om(\htl)$; thus 
\begin{equation}\label{eq:definition-difference-quotient}
\kttl(z)=\frac{\Th(z)-\Th(\la)}{z-\la},\qquad 
\httl(z)=\sqrt{\frac{1-|\lambda|^2}{1-|\Theta(\lambda)|^2}}
\frac{\Th(z)-\Th(\la)}{z-\la}.
\end{equation}

In the sequel we will use the following simple lemmas. 

\begin{lem}\label{le:product}
Suppose $\Th_1, \Th_2$ are two inner functions, $f_1\in K_{\Th_1}$, 
$f_2\in K_{\Th_2}\cap H^\infty$. Then $f_1f_2, zf_1 f_2\in K_{\Th_1\Th_2}$.
\end{lem}

\begin{proof}
Obviously $z f_1 f_2\in H^2$. On the other side, 
$f_1\in K_{\Th_1}$ implies $f_1=\Th_1 \overline{ zg_1}$, 
with $g_1\in H^2$, and similarly $f_2=\Th_2 \overline{z g_2}$, 
$g_2\in H^\infty$. Thus $zf_1f_2\in\Th_1\Th_2 \overline{z H^2}$. 
Therefore $z f_1 f_2\in H^2\cap\Th_1\Th_2 \overline{H^2_0}=K_{\Th_1\Th_2}$. 
The claim about $f_1 f_2$ is an immediate consequence, since the model 
spaces are invariant under the backward shift operator $S^*$. 
\end{proof}

Recall that, given two inner functions $\theta_1$, $\theta_2$, we say that $\theta_2$ divides $\theta_1$ if there exists an inner function $\theta_3$ such that $\theta_1=\theta_2\theta_3$.

\begin{lem}\label{le:technique-espace-modele}
Suppose that $\theta$ and $\Theta$ are two inner functions such that $\theta^3$ divides $z\Theta$. Then:

{\rm (a)}  $\theta K_\theta\subset K_{\theta^2}\subset K_\Theta$;
 
{\rm (b)} if $f\in H^\infty\cap \theta K_\theta$ and $\phi\in K_\theta+\overline{K_\theta}$, then the functions $\phi f$ and $\bar\phi f$ belong to $K_\Theta$.
\end{lem}

\begin{proof} Since $\theta^3$ divides $z\Theta$, there exists an inner function $\theta_1$ such that $z\Theta=\theta^3\theta_1$. In particular it follows from this factorization that $\theta(0)\theta_1(0)=0$, which implies that $\theta\theta_1 H^2\subset zH^2$.

Using $K_\theta=H^2\cap \theta\,\overline{zH^2}$, we have
\[
\theta K_\theta =\theta H^2\cap \theta^2\,\overline{zH^2}\subset H^2\cap \theta^2\,\overline{zH^2}=K_{\theta^2}.
\]
Further,
\[
K_{\theta^2}=H^2\cap \theta^2\,\overline{z H^2}=H^2\cap \Theta \overline{z\Theta}\theta^2\,\overline{H^2}=H^2\cap \Theta \,\overline{\theta\theta_1 H^2}\subset H^2\cap \Theta \,\overline{z H^2}=K_\Theta, 
\]
because $\theta\theta_1 H^2\subset zH^2$; thus (a) is proved.

Let now $f=\theta f_1$ and $\phi=\phi_1+\overline{\phi_2}$, with $f_1\in H^\infty\cap K_\theta$ and $\phi_1,\phi_2\in K_\theta$. Since $\phi_2\in K_\theta$, using Lemma~\ref{le:omega}, we have $\phi_2=\theta \bar z \overline{\tilde \phi_2}$, with $\tilde\phi_2\in K_\theta$, which implies that 
\[
\phi f=\theta f_1(\phi_1+\overline{\phi_2})=\theta f_1 \phi_1+z f_1  \tilde\phi_2.
\]
But it follows from Lemma~\ref{le:product}  that $z f_1  \tilde\phi_2\in K_{\theta^2}$;  by (a), we obtain $z f_1  \tilde\phi_2\in K_\Theta$. So it remains to prove that $\theta f_1 \phi_1\in K_\Theta$. Obviously $\theta f_1 \phi_1\in H^2$; moreover, for every function $h\in H^2$, we have
\[
\langle \theta f_1\phi_1,\Theta h \rangle=\langle z\theta f_1\phi_1,z\Theta h \rangle=\langle z\theta f_1\phi_1,\theta^3\theta_1 h \rangle=\langle zf_1\varphi_1,\theta^2\theta_1 h\rangle=0,
\]
because another application of Lemma~\ref{le:product} yields $zf_1\varphi_1\in K_{\theta^2}$. That proves that $\theta f_1 \phi_1\in K_\Theta$ and thus $\varphi f\in K_\Theta$. Since $K_\Theta+\overline{K_\Theta}$ is invariant under the 
conjugation, we obtain also the result for $\bar\varphi f$.
\end{proof}

\subsection{Angular derivatives and evaluation on the boundary}\label{sse:angular}

The inner function $\Theta$ is said to have  an \emph{angular 
derivative in the sense of Carath\'eodory} at $\zeta\in \TTT$ if $\Theta$ and $\Theta'$ have a non-tangential limit at $\zeta$ and $|\Theta(\zeta)|=1$. Then it is known~\cite{AC70} that evaluation at $\zeta$ is continuous on $K_\Theta$, and the function $k^\Theta_\zeta$, defined by 
\[
 k^\Theta_\zeta(z):=\frac{1-\overline{\Theta(\zeta)}\Theta(z)}{1-\bar\zeta z},\qquad z\in\DDD,
\]
belongs to $K_\Theta$ and is the corresponding reproducing kernel. Replacing $\lambda$ by $\zeta$ in the formula~\eqref{eq:definition-difference-quotient} gives a function $\tilde k_\zeta^\Theta$ which also belongs to $K_\Theta$ and  $\omega(k_\zeta^\Theta)=\tilde k_\zeta^\Theta=\bar\zeta \Theta(\zeta) k_\zeta^\Theta$. 
Moreover we have $\|k_\zeta^\Theta\|_2=|\Theta'(\zeta)|^{1/2}$. We denote by $E(\Theta)$ 
the set of points $\zeta\in\TTT$ where $\Theta$ 
has an angular derivative in the sense of Carath\'eodory.

In~\cite{AC70} and~\cite{cohn3} precise conditions 
are given for the inclusion of $k^\Theta_\zeta$ into $L^p$ (for $1<p<\infty$); 
namely, if $(a_k)$ are the zeros of $\Theta$ in $\DDD$ and $\sigma$ 
is the singular measure on $\TTT$ corresponding to the singular part
of $\Theta$, then $k^\Theta_\zeta\in L^p$ if and only if
\begin{equation}\label{eq:cohn}
\sum_{k}\frac{1-|a_k|^2}{|\zeta-a_k|^p}+
\int_\TTT \frac{d\sigma(\tau)}{|\zeta-\tau|^p}< +\infty.
\end{equation}

We will use in the sequel the following easy result.
\begin{lem}\label{Lem:comparaison-Theta-Theta2}
Let $1<p<+\infty$ and let $\Theta$ be an inner function. Then we have:
\begin{enumerate}
 \item[$\mathrm{(a)}$] $E(\Theta^2)=E(\Theta)$;
 \item[$\mathrm{(b)}$] $\displaystyle\inf_{\lambda\in \DDD \cup 
      E(\Theta)}\|k_\lambda^\Theta\|_2>0$;
 \item[$\mathrm{(c)}$] 
      for $\lambda\in \DDD$, we have
\begin{equation}\label{eq:noyau-reproduisant-square}
 C\|k_\lambda^\Theta\|_p\leq \|k_\lambda^{\Theta^2}\|_p\leq 2 \|k_\lambda^\Theta\|_p,
\end{equation}
where $C=\|P_\Theta\|_{L^p\to L^p}^{-1}$ is a constant which depends only on $\Theta$ and $p$.
Also, if $\zeta\in E(\Theta)$, then $k_\zeta^{\Theta^2}\in L^p$ 
if and only if $k_\zeta^\Theta\in L^p$, and 
\eqref{eq:noyau-reproduisant-square} holds for $\lambda=\zeta$. 
\end{enumerate}
\end{lem}
\begin{proof} The proof of $\mathrm{(a)}$ is immediate using the definition. 
For the proof of $\mathrm{(b)}$ note that, for $\lambda\in 
\DDD \cup E(\Theta)$, we have
\[
 |1-\overline{\Theta(0)}\Theta(\lambda)| = |k_0^\Theta(\lambda)| \le 
 \|k^\Theta_0\|_2\|k^\Theta_\lambda\|_2 = (1-|\Theta(0)|^2)^{1/2} \|k^\Theta_\lambda\|_2,
\]
which implies $\|k^\Theta_\lambda\|_2 \ge \Big(\frac{1-|\Theta(0)|}
{1+|\Theta(0)|}\Big)^{1/2}$.  

It remains to prove $\mathrm{(c)}$.  
We have $k_\lambda^{\Theta^2}=(1+\overline{\Theta(\lambda)}\Theta)k_\lambda^\Theta$, 
whence $P_\Theta k_\lambda^{\Theta^2}=k_\lambda^\Theta$. Thus the result follows 
from the fact that $P_\Theta$ is bounded on $L^p$ and from the trivial estimate 
$|1+\overline{\Theta(\lambda)}\Theta(z)|\leq 2$, $z\in\TTT$.  
\end{proof}

\subsection{The continuous case}\label{sse:continuous}

It is useful to remember the connection with the ``continuous" case, for which we refer to \cite{duren,hnp}. If $u(w)=\frac{w-i}{w+i}$, then $u$ is a conformal homeomorphism of the Riemann sphere. It maps $-i$ to $\infty$, $\infty$ to 1, $\RRR$ onto $\TTT$ and $\CCC_+$ to $\DDD$ (here $\CCC_+=\{z\in\CCC : \Im z >0\}$).

The operator
\begin{equation*}
(\mathcal{U}f)(t)=\frac{1}{\sqrt\pi(t+i)} f(u(t))
\end{equation*}
maps $L^2(\TTT)$ unitarily onto $L^2(\RRR)$ and $H^2$ unitarily onto $H^2(\CCC_+)$, the Hardy space of the upper
half-plane. The corresponding transformation for functions in
$L^\infty$ is
\begin{equation}
\tilde{\mathcal{U}}(\phi)= \phi\circ u;
\end{equation}
it maps $L^\infty(\TTT)$ isometrically onto $L^\infty(\RRR)$, 
$H^\infty$ isometrically onto $H^\infty(\CCC_+)$ and inner functions in $\DDD$ into inner functions in $\CCC_+$.
Now if $\Theta$ is an  inner function in $\DDD$,  we have  $\mathcal{U}P_\Theta=\bm{P_{\bm\Theta}}\mathcal{U}$ and then $\mathcal{U}K_\Theta=\bm{K_{\bm\Theta}}$, where $\bm{\Theta}=\Theta\circ u$,  $\bm{K_\Theta}=H^2(\CCC_+)\ominus\bm{\Theta}H^2(\CCC_+)$ and $\bm{P_{\bm\Theta}}$ is the orthogonal projection onto $\bm{K_\Theta}$. Moreover
\begin{equation}\label{transfer-disque-demi-plan-noyau}
\mathcal{U}h_\lambda^\Theta=c_\mu \bm{h_\mu^\Theta}\quad\hbox{and}\quad \mathcal{U}\tilde h_\lambda^\Theta=\overline{c_\mu}\bm{\tilde h_\mu^\Theta},
\end{equation}
where $\mu=u^{-1}(\lambda)\in\CCC_+$, $c_\mu=\frac{\bar\mu-i}{|\mu+i|}$ is a constant of modulus one, 
\[
\bm{h_\mu^\Theta}(\omega)=\frac{i}{\sqrt\pi}\sqrt{\frac{\Im\mu}{1-|\bm{\Theta}(\mu)|^2}}\frac{1-\overline{\bm\Theta(\mu)}\bm\Theta(\omega)}{\omega-\bar\mu},\qquad \omega\in\CCC_+,
\]
is the normalized reproducing kernel for $\bm{K_\Theta}$, while 
\[
\bm{\tilde h_\mu^\Theta}(\omega)=\frac{1}{i\sqrt\pi}\sqrt{\frac{\Im\mu}{1-|\bm{\Theta}(\mu)|^2}}\frac{\bm\Theta(\omega)-\bm\Theta(\mu)}{\omega-\mu},\qquad \omega\in\CCC_+,
\]
is the normalized difference quotient in $\bm{K_\Theta}$.

\mysection{Truncated Toeplitz operators}\label{se:Toeplitz}
In \cite{Sarason}, 
D. Sarason studied the class of {\em truncated Toeplitz operators} 
which are defined as the compression of Toeplitz operators 
to coinvariant subspaces of $H^2$.

Note first that we can extend the definitions of $M_\phi$, $T_\phi$, and $H_\phi$ in Section~\ref{se:preliminaries} to the case when the symbol is only in $L^2$ instead of $L^\infty$, obtaining (possibly unbounded) densely defined operators. Then $M_\phi$ and $T_\phi$ are bounded if and only if $\phi\in L^\infty$ (and 
$\|M_\phi\|=\|T_\phi\|=\|\phi\|_\infty$), while $H_\phi$ is bounded if 
and only if $P_-\phi\in BMO$ (and $\|H_\phi\|$ is equivalent to $\|P_-\phi\|_{BMO}$).

In \cite{Sarason}, D. Sarason defines an analogous operator on $K_\Th$.
Suppose $\Theta$ is an inner function and $\phi\in L^2$; the \emph{truncated Toeplitz operator} $A_\phi^\Theta$ will
in general be a densely defined, possibly unbounded, operator on 
$K_\Theta$. Its domain is $K_\Theta\cap H^\infty$, on which it acts by the formula
\[
A_\phi^\Theta f =P_\Theta (\phi f),\qquad f\in K_\Theta\cap H^\infty.
\]
In particular, $K_\Theta\cap H^\infty$ contains all reproducing kernels 
$k_\lambda^\Theta$, $\lambda \in \DDD$, 
and their linear combinations, and is therefore dense in $K_\Theta$.

We will denote by $\TT(K_\Theta)$ the space of all bounded 
truncated Toeplitz operators on $K_\Theta$. 
It follows from \cite[Theorem 4.2]{Sarason} that $\TT(K_\Theta)$ is a Banach space
in the operator norm.

Using Lemma~\ref{le:omega} and the fact that 
$\omega M_\phi \omega=M_{\bar\phi}$, it is easy to check the  useful formula
\begin{equation}\label{eq:omToom}
\om A_\phi^\Theta \om=A_{\bar\phi}^\Theta=(A_\phi^\Theta)^*.
\end{equation}

We call $\phi$ a \emph{symbol} of the operator $A_\phi^\Theta$. It is not unique; in \cite{Sarason}, it is shown that $A_\phi^\Theta=0$ if and only if $\phi\in  \Theta H^2+\overline{\Theta H^2}$. Let us denote $\mathfrak{S}_\Theta=L^2\ominus ( \Theta H^2+\overline{\Theta H^2})$ and $P_{\mathfrak{S}_\Theta}$ the corresponding orthogonal projection. Two spaces that contain $\mathfrak{S}_\Theta$ up to a subspace of dimension at most~1 admit a direct description, and we will gather their properties in the next two lemmas.

\begin{lem}\label{le:inclusion1} Denote by $Q_\Theta$ the orthogonal projection onto 
$K_\Theta \oplus\bar z\overline{K_\Theta}$. Then:

{\rm (a)} $Q_\Theta(\bar\Theta)=\bar\Theta-\overline{\Th(0)}^2\Th$;

{\rm (b)} we have
\[
K_\Theta \oplus\bar z\overline{K_\Theta}= \mathfrak{S}_\Theta\oplus\CCC q_\Th,
\]
where $q_\Th=\|Q_\Theta(\bar\Theta)\|_2^{-1}Q_\Theta(\bar\Theta)$; 

{\rm (c)} 
$Q_\Th$ and $P_{\mathfrak{S}_\Theta}$ are bounded on $L^p$ for $1<p<\infty$.
\end{lem}

\begin{proof}
Since by Lemma~\ref{le:omega} $\bar z\overline{K_\Theta}=\bar \Th K_\Th$, we have $K_\Th\oplus \bar z\overline{K_\Theta}=K_\Th \oplus \bar \Th K_\Th$, and therefore $Q_\Th=P_\Th+M_{\bar\Th}P_\Th M_\Th$. Thus $Q_\Th$ is bounded on $L^p$ for all $p>1$. Further, if we denote by $\bm{1}$ the constant function equal to $1$, then
\[
\begin{split}
Q_\Theta(\bar\Theta)&= P_\Th (\bar\Theta) +M_{\bar\Th}P_\Th M_\Th (\bar\Theta) 
= P_\Th(\overline{\Th(0)}\bm{1}) + M_{\bar\Th}P_\Th \bm{1}\\
&=
(\overline{\Th(0)} +\bar\Th)(1-\overline{\Th(0)}\Th)= \bar\Theta-\overline{\Th(0)}^2\Th.
\end{split}
\]
Thus (a) is proved.

Since $L^2=\Theta H^2\oplus \overline{\Theta H^2_0} \oplus K_\Theta 
\oplus\bar z\overline{K_\Theta}$, it follows that $\mathfrak{S}_\Theta\subset K_\Theta\oplus \bar z\overline{K_\Theta}$ and thus
\begin{equation}\label{eq:Q_Theta}
K_\Theta \oplus\bar z\overline{K_\Theta}=Q_\Theta\left(\mathfrak{S}_\Theta+\Theta H^2+\overline{\Theta H^2_0}+\CCC\bar\Theta\right)=\mathfrak{S}_\Theta\oplus\CCC Q_\Theta(\bar\Theta),
\end{equation}
which proves (b). Note that according to (a), one easily see that $Q_\Theta(\bar\Theta)\not\equiv 0$.

Now we have for $f\in L^\infty$
\begin{equation}\label{eq:P_S}
P_{\mathfrak{S}_\Theta}f= Q_\Th f- \<f,q_\Th\>q_\Th.
\end{equation}
and the second term is bounded in $L^p$, since $q_\Th$ belongs
to $L^\infty$. This concludes the proof of (c).
\end{proof}

\begin{lem}\label{le:inclusion2}
We have $\mathfrak{S}_\Theta\subset  K_\Theta+\overline{K_\Theta}$. 
Each truncated Toeplitz operator has a symbol $\phi$ of the 
form $\phi=\phi_+ + \overline{\phi_-}$ with $\phi_\pm\in K_\Theta$; 
any other such decomposition corresponds to $\phi_+ +c k^\Theta_0$, 
$\phi_--\bar c k^\Theta_0$ for some $c\in\CCC$. In particular, 
$\phi_\pm$ are uniquely determined if we fix (arbitrarily) the 
value of one of them in a point of $\DDD$.
\end{lem}

\begin{proof}
See \cite[Section 3]{Sarason}.
\end{proof}

The formulas $\psi=\lim_{n\to\infty}\bar z^n T_\psi(z^n)$ and $P_-\psi=H_\psi(\bm{1})$ allow one to recapture simply the unique symbol of a Toeplitz operator as well as the unique symbol in $H^2_-$ of a
Hankel operator.
It is interesting to obtain a similar 
direct formula for the symbol of a truncated Toeplitz operator. 
Lemma~\ref{le:inclusion2} says that the symbol is unique if we assume, for instance, that 
$\phi=\phi_++\overline{\phi_-}$, with $\phi_\pm\in \Kth$ and 
$\phi_-(0)=0$. We can then recapture $\phi$ from the action of $A_\phi^\Theta$ on 
$k^\Theta_\lambda$ and $\tilde k^\Theta_\lambda$. Indeed, one can check that
\begin{equation}\label{eq:recap0}
\begin{split}
\Ath_\phi k^\Theta_0&=\phi_+  -\overline{\Theta(0)}\,\Theta\, \overline{\phi_-},\\
\Ath_\phi\tilde k^\Theta_0&=\omega \left(  \phi_- +\overline{\phi_+(0)} -\overline{\Theta(0)}\,\Theta\,\overline{\phi_+} \right).
\end{split}
\end{equation}
From the first equation we obtain $\phi_+(0)=\<\Ath_\phi k^\Theta_0,  k^\Theta_0\>$. 
Then~\eqref{eq:recap0} imply, for any $\lambda\in\DDD$, 
\begin{equation*}
\begin{split}
\phi_+(\lambda)  -\overline{\Theta(0)}\Theta(\lambda)\overline{\phi_-(\lambda)}&
=\<\Ath_\phi k^\Theta_0,k^\Theta_\lambda\> ,\\
\overline{\phi_-(\lambda)}  -\Theta(0)\overline{\Theta(\lambda)}\phi_+(\lambda)  &=\<\Ath_\phi \tilde k_0^\Theta,\tilde k_\lambda^\Theta\>-\<\Ath_\phi k_0^\Theta,k_0^\Theta\>.
\end{split}
\end{equation*}
This is a linear system in $\phi_+(\lambda)$ and $\overline{\phi_-(\lambda)}$, whose determinant is $1-|\Theta(0)\Theta(\lambda)|^2>0$; therefore, $\phi_\pm$ can be made explicit in terms of the products in the right hand side.

Note, however, that  $\Ath_\phi$ is completely determined by 
its action on reproducing kernels, so one should be able to recapture the values of the symbol only from  $\Ath_\phi k^\Theta_\lambda$. 
The next proposition shows how one can achieve this goal; moreover, one can also obtain an estimate of the $L^2$-norm of the symbol. Namely, 
for an inner function $\Theta$ and any (not necessarily bounded) linear operator $T$ whose domain contains $K_\Theta\cap H^\infty$, define 
\begin{equation}\label{eq:rho}
\rho_r(T):= \sup_{\lambda\in\DDD}\|T h^\Theta_\lambda\|_2.
\end{equation}
We will have the occasion to come back to the quantity $\rho_r$ in the next section.

To simplify the next statement, denote 
\begin{equation}\label{eq:F}
F_{\lambda,\mu}= (I-\lambda S^*) \omega(A^\Theta_\phi k^\Theta_\lambda) 
- (I-\mu S^*) \omega(A^\Theta_\phi k^\Theta_\mu),\qquad \lambda,\mu\in\DDD.
\end{equation}

\begin{prop}\label{pr:recapturing} 
Let $\Theta$ be an inner function,  $A_\phi^\Theta$  a truncated Toeplitz operator, 
and $\mu\in\DDD$ such that $\Theta(\mu)\not=0$. Suppose $\phi=\phi_++\overline{\phi_-}$ is the unique 
decomposition of the symbol
with $\phi_\pm\in K_\Theta$, $\phi_-(\mu)=0$. 
Then
\begin{equation}\label{eq:recap4}
\phi_-(\lambda)=\frac{ \< (S-\mu)(I-\mu S^*)^{-1} F_{\lambda,\mu}, 
k^\Theta_\mu\>}{\Theta(\mu) (\overline{\Theta(0)}\Theta(\mu)-1)},\qquad \lambda\in\DDD,
\end{equation}
and $\phi_+=\omega(\psi_+)$, where
\begin{equation}\label{eq:psi}
\psi_+= (I-\mu S^*) \omega(A^\Theta_\phi k^\Theta_\mu)+ \Theta(\mu) S^* \phi_-. 
\end{equation}
Moreover, there exists a
constant $C$ depending only on $\Theta$ and $\mu$ such that 
\begin{equation}\label{eq:phi-rho}
\max\{\|\phi_{-}\|_2, \|\phi_{+}\|_2\}\le C \rho_r(\Ath_\phi).
\end{equation}
\end{prop}

\begin{proof}
First note that for any
$\lambda\in \DDD$, we have
\begin{equation}\label{eq:recap1}
(I-\lambda S^*) \omega(A^\Theta_\phi k^\Theta_\lambda) = 
\psi_+ +\phi_-(\lambda)S^* \Theta -\Theta(\lambda) S^* \phi_-.
\end{equation}
Indeed,
\[
\begin{aligned}
P_\Theta (\phi_+ k^\Theta_\lambda) & = 
P_\Theta \bigg(\phi_+ \frac{1}{1-\overline \lambda z}\bigg) = 
\phi_+ +
\bar\lambda P_\Theta \bigg(\frac{\Theta \overline{z} \overline{\psi_+}}
{\overline z - \overline \lambda}\bigg) \\
& = 
\phi_+ +
\bar\lambda \Theta \overline{z}  \overline{\bigg(\frac{\psi_+ - \psi_+(\lambda)}
{z - \lambda}\bigg)}.
\end{aligned}
\]
Thus,
\[
\omega (A^\Theta_{\phi_+} k^\Theta_\lambda) = \psi_+
+ \lambda \frac{\psi_+ - \psi_+(\lambda)}{z - \lambda}=\frac{z\psi_+-\lambda\psi_+(\lambda)}{z-\lambda}.
\]

One can easily check that 
\begin{equation}\label{eq:formule-standard-shift}
(I-\lambda S^*)^{-1}S^*f=\frac{f-f(\lambda)}{z-\lambda},
\end{equation} 
for every function $f\in H^2$; then we obtain
\begin{equation}\label{eq:recap1-1}
(I-\lambda S^*) \omega(A^\Theta_{\phi_+} k^\Theta_\lambda) = \psi_+.
\end{equation}
On the other hand,
\[
\begin{aligned}
P_\Theta (\phi_- k^\Theta_\lambda) & =
P_\Theta\bigg(\overline z \frac{\overline{\phi_-} - \overline{\phi_-(\lambda)}}
{\overline z - \overline \lambda} + \frac{\overline{\phi_-(\lambda)}}
{1-\overline \lambda z}        \\
& - \overline{\Theta (\lambda)} \overline z \Theta
\frac{\overline{\phi_-} - \overline{\phi_-(\lambda)}}
{\overline z - \overline \lambda} - \overline{\Theta (\lambda)} \Theta
\frac{\overline{\phi_-(\lambda)}}
{1-\overline \lambda z}  \bigg)  \\
& = \overline{\phi_-(\lambda)}k^\Theta_\lambda - 
\overline{\Theta (\lambda)} \overline z \Theta
\overline{ \bigg(\frac{\phi_- - \phi_-(\lambda)}
{z - \lambda}\bigg)}.
\end{aligned}
\]
Hence,
\[
\omega (A^\Theta_{\phi_-} k^\Theta_\lambda) = 
\phi_-(\lambda) \frac{\Theta -\Theta(\lambda)}{z-\lambda} - 
\Theta(\lambda) \frac{\phi_--\phi_-(\lambda)}{z-\lambda}
\]
and
\begin{equation}\label{eq:recap1-2}
(I-\lambda S^*) \omega(A^\Theta_{\phi_-} k^\Theta_\lambda) =
\phi_-(\lambda) S^* \Theta  - \Theta(\lambda) S^* \phi_-.
\end{equation}
Thus \eqref{eq:recap1} follows immediately from \eqref{eq:recap1-1} and \eqref{eq:recap1-2}.
If we take $\lambda=\mu$ in (\ref{eq:recap1}), 
we get (remembering that $\phi_-(\mu)=0$)
\begin{equation}\label{eq:recap2-bis}
\psi_+= (I-\mu S^*) \omega(A^\Theta_\phi k^\Theta_\mu)+ \Theta(\mu) S^* \phi_-.
\end{equation}
Now plugging~\eqref{eq:recap2-bis} into~\eqref{eq:recap1} yields 
\begin{equation*}
\phi_-(\lambda)S^* \Theta +(\Theta(\mu) -\Theta(\lambda)) 
S^* \phi_-= F_{\lambda,\mu}.
\end{equation*}
Therefore, applying  $(S-\mu)(I-\mu S^*)^{-1}$ and using $\phi_-(\mu)=0$ and \eqref{eq:formule-standard-shift}, we obtain
\begin{equation}\label{eq:recap3}
\phi_-(\lambda)(\Theta-\Theta(\mu)) + (\Theta(\mu) -\Theta(\lambda))  \phi_-=(S-\mu)(I-\mu S^*)^{-1} F_{\lambda,\mu}.
\end{equation}
Finally, we take the scalar product of both sides with $k^\Theta_\mu$ and 
use the fact that $\Theta\perp K_\Theta$, $P_\Theta \bm{1}=\bm{1}-\overline{\Theta(0)}\Theta$, 
and again $\phi_-(\mu)=0$. Therefore
\[
-\phi_-(\lambda)\Theta(\mu) (1-\overline{\Theta(0)}\Theta(\mu))= 
\< (S-\mu)(I-\mu S^*)^{-1} F_{\lambda,\mu}, k^\Theta_\mu\>,
\]
which immediately implies \eqref{eq:recap4}.

To obtain the boundedness of the $L^2$ norms, fix now $\lambda\in\DDD$ such that $\Th(\lambda)\not=\Th(\mu)$.
Since
\[
\|(I-\mu S^*) \omega(A^\Theta_\phi k^\Theta_\mu)\|_2
\le 2\|A^\Theta_\phi k^\Theta_\mu\|_2 
\le 2 \|k_\mu^\Theta\|_2  \rho_r(A_\phi^\Theta)
\]
and a similar estimate holds for $\|(I-\lambda S^*) \omega(A^\Theta_\phi k^\Theta_\lambda)\|_2$, we have $\|F_{\lambda,\mu}\|_2\le C_1\rho(\Ath_\phi)$, where $C_1$, as well as the next constants appearing in this proof, depends only on $\Theta$, $\lambda$, $\mu$.
By~\eqref{eq:recap3}, it follows that  
\[
\|\phi_-(\lambda)(\Theta-\Theta(\mu)) + (\Theta(\mu) -\Theta(\lambda))  \phi_-\|_2 \le 
C_2\rho_r(A_\phi^\Theta).
\]
Projecting onto $K_\Theta$ decreases the norm; since $P_\Th(\phi_-(\lambda)\Theta)=0$ and $P_\Th(\bm{1})=k^\Th_0$, we obtain
\[
\| -\Theta(\mu)
\phi_-(\lambda)k_0^\Theta+(\Theta(\mu) - \Theta(\lambda))\phi_-\|_2 \le C_2\rho_r(A^\Theta_\phi).
\]

Write now $\phi_- = h+c k_0^\Theta$ with $h\perp k_0^\Theta$.
Then
$
\|(\Theta(\mu) - \Theta(\lambda))h\|_2
\le C_2\rho_r(A_\phi^\Theta )$, whence
$
\|h\|_2\le C_3\rho_r(A_\phi^\Theta )$.
Since $\phi_- (\mu ) = 0$, we have $h(\mu) + ck_0^\Theta(\mu) = 0$,
which implies that 
\[
 |c| =|k_0^\Theta(\mu)|^{-1 }|h(\mu)| \le C_4
 \rho_r(A_\phi^\Theta)
\]
Therefore we have $\|\phi_-\|_2\le C_5\rho_r(\Ath_\phi)$. Finally,~(\ref{eq:psi}) yields a similar estimate for~$\psi_+$ and then for~$\phi_+$.
\end{proof}

%
%
%
%

The following proposition yields a relation between truncated Toeplitz operators and usual 
Hankel operators.

\begin{prop}
With respect to the decompositions $H^2_-=\bar\Th K_\Th\oplus \bar\Th H_-^2$, 
$H^2=K_\Th\oplus \Th H^2$, the 
operator $H_{\bar\Th}^* H_{\bar\Th \phi}H_{\bar\Th}^* :H^2_-\to H^2$ 
has the matrix
\begin{equation}\label{eq:matrixH}
\begin{pmatrix}
A_\phi^\Theta M_\Th & 0\\ 0 & 0
\end{pmatrix}.
\end{equation}
\end{prop}

\begin{proof}
If $f\in \bar\Th H_-^2$, then $H_{\bar\Th}^* f=0$. If $f\in \bar\Th K_\Th$, then $H_{\bar\Th}^* f=\Th f\in K_\Th$. Since $P_\Th=P_+M_\Th P_- M_{\bar\Th}$, it follows that,
for $f\in K_\Th$, 
\[
A_\phi^\Theta f=P_\Th M_\phi  f= P_+M_\Th P_- M_{\bar\Th}M_\phi f= H_{\bar\Th}^* H_{\bar\Theta\phi}f,
\]
and therefore, if $f\in \bar\Th K_\Th$, then $A_\phi^\Theta \Th f= H_{\bar\Th}^* H_{\bar\Th \phi}H_{\bar\Th}^* f$ as required.
\end{proof}

The non-zero entry in~\eqref{eq:matrixH} consists in the isometry $M_\Th:\bar\Th K_\Th\to K_\Th$, followed by $A_\phi^\Theta$ acting on $K_\Th$. There is therefore a close connection between properties of $A_\phi^\Theta$ and properties of the corresponding 
product of three Hankel operators. Such products of Hankel operators have been studied for instance in \cite{Axler-Sarason,Brown-Halmos,Xia-Zheng}.

\begin{rem}\label{re:halfspace}
Truncated Toeplitz operators can be defined also on model spaces of 
$H^2(\CCC_+)$, that is, $\bm{K}_{\bm\Theta}= H^2(\CCC_+)\ominus \bm\Theta  
H^2(\CCC_+)$ for an inner function $\bm\Theta$ in the upper half-plane 
$\CCC_+$. We start then with a symbol $\bm \phi\in (t+i)L^2(\RRR)$ 
(which contains $L^\infty(\RRR)$) and define (for $f\in\bm{K_\Theta}
\cap (z+i)^{-1}H^\infty(\CCC_+)$, a dense 
subspace of $\bm{K_\Theta}$) the truncated Toeplitz operator 
$\bm A^{\bm \Theta}_{\bm \phi}f=P_{{\bm\Theta}} (\bm \phi f)$.  

Let us briefly explain the relations between the truncated Toeplitz operators corresponding to model spaces on the upper half-plane and those corresponding to model spaces on the unit disk. If $\Theta=\bm{\Theta}\circ u^{-1}$ and $\psi=\phi\circ u^{-1}$, using the fact that $\mathcal{U}P_\Theta\mathcal{U}^*=P_{{\bm\Theta}}$ and $\mathcal{U}M_\psi=M_\phi\mathcal{U}$, we easily obtain 
\[
\bm A^{\bm \Theta}_{\bm \phi}=\mathcal{U}A_\psi^\Theta\mathcal{U}^*.
\]
In particular, if $\bm{A}$ is a linear operator on $\bm{K_\Theta}$, then $\bm{A}$ 
is a truncated Toeplitz operator on $\bm{K}_{\bm\Theta}$ 
if and only if $A=\UU^* \bm{A} \UU$ is a truncated Toeplitz operator 
on $\Kth$, and $\bm{\phi}$ is a symbol for $\bm{A}$ if and only 
if $\psi:=\phi\circ u^{-1}$ is a symbol for $A$. It follows that 
$\bm{A}$ is bounded (or has a bounded symbol) if and only if  
$A$ is bounded (respectively, has a bounded symbol). Moreover 
we easily deduce from \eqref{transfer-disque-demi-plan-noyau} that
\[
\|\bm{A_\phi^\Theta}\bm{h_\mu^\Theta}\|_2=\|A_\psi^\Theta h_\lambda^\Theta\|_2\quad\hbox{and}\quad \|\bm{A_\phi^\Theta}\bm{\tilde h_\mu^\Theta}\|_2=\|A_\psi^\Theta \tilde h_\lambda^\Theta\|_2,
\]
for every $\mu\in\CCC_+$ and $\lambda=u(\mu)$. Finally, the truncated Toeplitz operator $\bm A^{\bm \Theta}_{\bm \phi}=0$ if and only if the symbol $\bm\phi \in (t+i)\left (\bm\Theta{H^2(\CCC_+)}\oplus \overline{\bm\Theta H^2(\CCC_+)}\right)$ (note that the  sum is in this case orthogonal, since $H^2(\CCC_+)\perp \overline{H^2(\CCC_+)}$).
\end{rem}

\mysection{Existence of bounded symbols and the Reproducing Kernel Thesis}

As noted in Section~\ref{se:Toeplitz}, a Toeplitz operator $T_\phi$ has a unique symbol,  $T_\phi$ is bounded if and only if this symbol is in $L^\infty$, and the map $\phi\mapsto T_\phi$ is isometric from $L^\infty$ onto the space of bounded Toeplitz operators on $H^2$. The situation is more complicated for Hankel operators: there is no uniqueness of the symbol, while the map $\phi\mapsto H_\phi$ is contractive and onto from $L^\infty$ to the space of bounded Hankel operators (the boundedness condition $P_-\phi\in BMO$ is equivalent to the fact that any bounded Hankel operator has a symbol in $L^\infty$).

In the case of truncated Toeplitz operators, the map $\phi\mapsto A_\phi^\Theta$
is again contractive from $L^\infty$ to $\TT(K_\Theta)$.
It is then natural to ask whether it is onto, that is, whether any 
bounded truncated Toeplitz operator is a compression of a bounded
Toeplitz operator in $H^2$.
This question has been asked by Sarason in~\cite{Sarason}.

\begin{question1}
Does every bounded truncated Toeplitz operator on $\Kth$ possess 
an  $L^\infty$ symbol?
\end{question1}

One may expect the answer to depend 
on the function $\Theta$, and indeed 
we show below that it is the case. Assume that for some inner function $\Theta$, any operator in $\TT(K_\Theta)$ has a bounded symbol. Then if follows from the open mapping theorem that there exists a constant $C$ such that for any $A\in \TT(K_\Theta)$ 
one can find $\phi\in L^\infty$ with $A=A^\Theta_\phi$ and $\|\phi\|_\infty\le C\|A\|$.

%
%

A second natural question that may be asked about truncated Toeplitz operators 
is the Reproducing Kernel 
Thesis (RKT). This is related to the quantity $\rho_r$ defined in~\eqref{eq:rho}. The functions $h_\lambda^\Theta$  have all norm~1, so
 if $\Ath_\phi$  
is bounded then obviously $\rho_r(\Ath_\phi)\le \|\Ath_{\phi}\|_2$. 
The following question is then natural:
\begin{question2}
(RKT for truncated Toeplitz operators): 
let $\Theta$ be an inner function and $\varphi\in L^2$. Assume that 
$\rho_r(A_\varphi^\Theta)<+\infty$. Is $A_\varphi^\Theta$  bounded on~$K_\Theta$?
\end{question2}

As we have seen in the introduction, the RKT is true for various classes of operators related to the truncated Toeplitz operators, and it seems natural to investigate it for this class. 
We will see in Section~\ref{se:negative} that the answer to this 
question is in general negative. 

As we will show below, 
it is more natural to restate the RKT by including in the 
hypothesis also the functions $\tilde h_\lambda^\Theta$. 
Thus, for any linear operator $T$  whose domain contains $K_\Theta\cap H^\infty$, define
\[
\rho_d(T)=\sup_{\lambda\in\DDD}\|T \tilde h_\lambda^\Theta\|_2,
\]
and $\rho(T)=\max\{\rho_r(T),\rho_d(T)\}$. 
The indices $r$ and $d$ in  notation $\rho_r$ and $\rho_d$ stand 
for ''reproducing kernels`` and ''difference quotients``.

Note that if $\Ath_\phi$ is a truncated Toeplitz operator, then by \eqref{eq:omToom}, 
we have $\rho_d(A_\phi^\Theta)=\rho_r((A_\phi^\Theta)^*)$, and then
\[
\rho(\Ath_\phi)=\max\{\rho_r(A_\phi^\Theta),\rho_r((A_\phi^\Theta)^*)\}.
\]

\begin{question3} Let $\Theta$ be an inner function and $\varphi\in L^2$. Assume that $\rho(\Ath_\phi)<\infty$. Is $A_\varphi^\Theta$  bounded on $K_\Theta$?
\end{question3}

In Section~\ref{se:negative}, we will show that the answer 
to Questions 1 and 2 may be negative. Question 3 remains in general open. In Section~\ref{mysection:positive-results}, we will give some examples of spaces $\Kth$ on which the answers to Questions 1 and 3 are positive.

In the rest of this section we will discuss the existence of bounded symbols and the RKT for some simple cases.


First, it is easy to deal with analytic or antianalytic symbols. The next proposition is a straightforward consequence of Bonsall's theorem \cite{Bon} and the commutant lifting theorem. The equivalence between $\mathrm{(i)}$ and $\mathrm{(ii)}$ has already been noticed in \cite{Sarason}.

\begin{prop}\label{pr:bonsall} 
Let $\phi\in H^2$ and let $A_\phi^\Theta$ be a truncated Toeplitz operator. Then the following assertions are equivalent:
\begin{enumerate}
 \item[$\mathrm{(i)}$] $A_\phi^\Theta$ has a bounded symbol;
\item[$\mathrm{(ii)}$] $A_\phi^\Theta$ is bounded;
\item[$\mathrm{(iii)}$] $\rho_r(A_\phi^\Theta)<+\infty$.
\end{enumerate}
More precisely there exists a universal constant $C>0$ such that any truncated Toeplitz operator $\Ath_\phi$ has a symbol $\phi_0$ with $\|\phi_0\|_\infty\le C\rho_r(\Ath_\phi)$.
\end{prop}

\begin{proof} It is immediate that $\mathrm{(i)}\Longrightarrow\mathrm{(ii)}\Longrightarrow\mathrm{(iii)}$. The implication $\mathrm{(ii)}\Longrightarrow\mathrm{(i)}$ has already been noted in \cite{Sarason}; indeed if $\phi\in H^2$ and $A_\phi^\Theta$ is bounded, then $A_\phi^\Theta$ commutes with $S_\Theta:=A_z^\Theta$ and then, by a corollary of the commutant lifting theorem, $A_\phi^\Theta$ has an $H^\infty$ symbol with norm equal to the norm of $A_\phi^\Theta$.  

So it remains to prove that there exists a constant $C>0$ such that $\|A_\phi^\Theta\|\leq C \rho_r(A_\phi^\Theta)$. If $f\in K_\Th\cap H^\infty$, then $\phi f\in H^2$. Therefore $P_\Th (\phi f)=
\Th P_-(\bar\Th \phi f)$, or, in other words,
$A_\phi^\Theta(f)=\Th H_{\bar\Th \phi}f$.

On the other hand, $\Theta H^2 \subset \ker H_{\bar\Th \phi}$, and therefore, with
respect to the decompositions $H^2= K_\Th\oplus \Th H^2$, $H^2_-=\bar\Th K_\Th\oplus 
\bar\Th H^2_-$, one can write 
\begin{equation}\label{eq:dec2}
H_{\bar\Th \phi}=
\begin{pmatrix}
\bar\Th A_\phi^\Theta & 0\\ 0 & 0
\end{pmatrix}.
\end{equation}
It follows that $A_\phi^\Theta$ is bounded if and only if $H_{\bar\Th \phi}$ is. By Bonsall's
Theorem \cite{Bon}, 
there exists a universal constant $C$ (independent of $\phi$ and $\Theta$) 
such that the boundedness of $H_{\bar\Th \phi}$ is equivalent to
$\sup_{\la\in\DDD}\|H_{\bar\Th \phi} h_\la\|_2<\infty$, and
\[
\|H_{\bar\Th \phi}\|\le C \sup_{\la\in\DDD}\|H_{\bar\Th \phi} h_\la\|_2.
\]
But, again by~\eqref{eq:dec2} and using \eqref{eq:klth} and \eqref{eq:cafe}, we have 
\[
H_{\bar\Th \phi} h_\la= \bar\Th A_\phi^\Theta P_\Th h_\la=\bar\Th(1-|\Th(\la)|^2)^{1/2} A_\phi^\Theta \htl, 
\]
and thus $\sup_{\la\in\DDD}\|H_{\bar\Th \phi} h_\la\|_2\le \sup_{\la\in\DDD}\|A_\phi^\Theta \htl\|_2=\rho_r(A_\phi^\Theta)$. The proposition is proved.
\end{proof}

A similar result is valid for antianalytic symbols.

\begin{prop}\label{pr:cobonsall} 
Let $\phi\in \overline{H^2}$ and let $A_\phi^\Theta$ be a truncated Toeplitz operator. Then the following assertions are equivalent:
\begin{enumerate}
 \item[$\mathrm{(i)}$] $A_\phi^\Theta$ has a bounded symbol;
\item[$\mathrm{(ii)}$] $A_\phi^\Theta$ is bounded;
\item[$\mathrm{(iii)}$] $\rho_d(A_\phi^\Theta)<+\infty$.
\end{enumerate}
More precisely there exists a universal constant $C>0$ such that any  truncated Toeplitz operator $\Ath_\phi$ has a symbol $\phi_0$ with $\|\phi_0\|_\infty\le C \rho_d(A_\phi^\Theta)$.

\end{prop}

\begin{proof}
Suppose  $\phi\in \overline{H^2}$.
Since $\|A_\phi^\Theta\|=\|(A_\phi^\Theta)^*\|=
\|A_{\bar\phi}^\Theta\|$, and $\bar\phi\in H^2$,
we may apply Proposition~\ref{pr:bonsall} to $A_{\bar\phi}^\Theta$ 
because by~\eqref{eq:omToom}, we have
\[
\rho_r(A_{\bar\phi}^\Theta)=\sup_{\la\in\DDD}\|A_{\bar\phi}^\Theta \htl\|_2 =\sup_{\la\in\DDD}\|A_{\phi}^\Theta\om \htl\|_2
=\sup_{\la\in\DDD}\|A_{\phi}^\Theta \httl\|_2=\rho_d(A_\phi^\Theta).\qedhere
\]
\end{proof}

As we have seen, if $\phi$ is bounded, then obviously the 
truncated Toeplitz operator $A_\phi^\Theta$ is bounded. We will see 
now that one can get a slightly more general result. It involves 
the so-called {\emph {Carleson curves}} associated with an inner 
function (see for instance \cite{garnett}). Recall that if $\Theta$ 
is an inner function and $\alpha\in (0,1)$, then the system of Carleson 
curves $\Gamma_\alpha$ associated to $\Theta$ and $\alpha$ is the countable 
union of closed simple and rectifiable curves in $\clos\DDD$ such that:
\begin{enumerate}
\item[$\bullet$] The interior of curves in $\Gamma_\alpha$ are pairwise disjoint.
\item[$\bullet$] There is a constant $\eta(\alpha)>0$ such that for every $z\in\Gamma_\alpha\cap\DDD$ we have
\begin{eqnarray}\label{eq:carleson-curve1}
\eta(\alpha)\leq |\Theta(z)|\leq\alpha.
\end{eqnarray}
\item[$\bullet$] Arclength $|dz|$ on $\Gamma_\alpha$ is a Carleson measure, which means that there is a constant $C>0$ such that 
\[
\int_{\Gamma_\alpha}|f(z)|^2\,|dz|\leq C \|f\|_2^2,
\]
for every function $f\in H^2$.
\item[$\bullet$] For every function $\phi\in H^1$, we have
\begin{eqnarray}\label{eq:cauchy-espace-modele}
\int_\TTT\frac{\phi(z)}{\Theta(z)}dz=\int_{\Gamma_\alpha}\frac{\phi(z)}{\Theta(z)}dz.
\end{eqnarray}
\end{enumerate}


\begin{prop}\label{Lem:cle-curve}
Let $\phi\in H^2$ and assume that $|\phi||dz|$ is a Carleson measure on $\Gamma_\alpha$. Then $A_\phi^\Theta$ is a bounded truncated Toeplitz operator on $K_\Theta$ and it has a bounded symbol.
\end{prop}

\begin{proof}
Let $f,g\in K_\Theta$ and assume further that $f\in H^\infty$. Then we have
\[
\langle A_\phi^\Theta f,g\rangle=\langle \phi f,g \rangle=\int_\TTT \phi(z)f(z)\overline{g(z)}dz.
\]
Since $g\in K_\Theta$, we can write (on $\TTT$), $g(z)=\bar z \overline{h(z)}\Theta(z)$, with $h\in K_\Theta$. Therefore
\[
\langle A_\phi^\Theta f,g\rangle=\int_\TTT \frac{z\phi(z)f(z)h(z)}{\Theta(z)}dz.
\]
But $zf(z)\phi(z)h(z)\in H^1$ and using \eqref{eq:cauchy-espace-modele}, we can write
\[
\langle A_\phi^\Theta f,g\rangle=\int_{\Gamma_\alpha} \frac{z\phi(z)f(z)h(z)}{\Theta(z)}dz.
\]
Therefore, according to \eqref{eq:carleson-curve1}, we have
\[
|\langle A_\phi^\Theta f,g\rangle|\leq \int_{\Gamma_\alpha}\frac{|z\phi(z)f(z)h(z)|}{|\Theta(z)|}|dz|\leq \frac{1}{\eta(\alpha)}\int_{\Gamma_\alpha}|f(z)||h(z)||\phi(z)||dz|.
\]
Hence, by the  Cauchy-Schwarz 
inequality and using the fact  that $|\phi||dz|$ is a Carleson measure on $\Gamma_\alpha$, we have
\[ 
|\langle A_\phi^\Theta f,g\rangle|\leq C\frac{1}{\eta(\alpha)}\|f\|_2\|g\|_2.
\]
Finally, we get that $A_\phi^\Theta$ is bounded. Since $\phi$ is analytic it follows from Proposition~\ref{pr:bonsall} that  $A_\phi^\Theta$ has a bounded symbol. 
\end{proof}

\begin{cor}\label{cor:condition-suffisante-bornitude}
Let $\phi=\phi_1+\overline{\phi_2}$, with $\phi_i\in H^2$, $i=1,2$. Assume that $|\phi_i||dz|$ are Carleson measures on $\Gamma_\alpha$ for $i=1,2$. Then $A_\phi^\Theta$ is bounded and has a bounded symbol. 
\end{cor}

\begin{proof}
Using Proposition~\ref{Lem:cle-curve}, we get immediately that $A_{\phi_i}^\Theta$ is bounded and has a bounded symbol $\widetilde{\phi_i}$, for $i=1,2$. Therefore, $A_{\overline{\phi_2}}^\Theta=(A_{\phi_2}^\Theta)^*$ is also bounded and has a bounded symbol $\overline{\widetilde\phi_2}$. Hence we get that $A_\phi^\Theta=A_{\phi_1}^\Theta+A_{\overline{\phi_2}}^\Theta$ is bounded and it has a bounded symbol, say $\widetilde{\phi_1}+\overline{\widetilde\phi_2}$.
\end{proof}

\begin{rem}
By the construction of the Carleson curves $\Gamma_\alpha$ associated to an inner function $\Theta$, we know that $|dz|$ is a Carleson measure on $\Gamma_\alpha$. Therefore, Proposition~\ref{Lem:cle-curve} can be applied if $\phi$ is bounded on $\Gamma_\alpha$ and Corollary~\ref{cor:condition-suffisante-bornitude} can be applied if $\phi_1$, $\phi_2$ are bounded on $\Gamma_\alpha$.

\end{rem}

\mysection{Counterexamples}\label{se:negative}

We will show that under certain conditions on the inner function~$\Th$ there 
exist rank one bounded truncated Toeplitz operators that 
have no bounded symbol.
It is proven in~\cite[Theorem 5.1]{Sarason} 
that any rank one truncated Toeplitz operator is either of the form
$k^\Theta_\lambda\otimes \tilde k^\Theta_\lambda$ or 
$\tilde k^\Theta_\lambda\otimes k^\Theta_\lambda$
for $\lambda\in \DDD$, or of the form 
$k^\Theta_\zeta\otimes k^\Theta_\zeta$ where $\zeta\in \TTT$
and $\Theta$ has an angular derivative at $\zeta$. 
In what follows we will use a representation
of the symbol of a rank one operator which differs slightly
from the one given in~\cite{Sarason}. 

\begin{lem}\label{le:anton1}
If $\lambda\in \DDD\cup E(\Theta)$, then  $\phi_\lambda=\Theta 
\bar z \overline{k^{\Theta^2}_\lambda}\in K_\Theta\oplus \bar z \overline{K_\Theta}$ 
is a symbol for $\tilde k^\Theta_\lambda\otimes {k}^\Theta_\lambda$. In particular, if $\zeta\in E(\Theta)$, then
$\phi_\zeta=\Theta \bar z \overline{k^{\Theta^2}_\zeta}$ is a symbol for 
$\Theta(\zeta) \overline\zeta \,k^\Theta_\zeta\otimes {k}^\Theta_\zeta$.  
\end{lem}

\begin{proof}   If $\zeta \in E(\Theta)$, then by 
Lemma~\ref{Lem:comparaison-Theta-Theta2}, $\Theta^2$  
has an angular derivative at $\zeta$, and so  
$k^{\Theta^2}_\zeta\in K_{\Theta^2}=K_\Theta\oplus \Theta K_\Theta$. 
It follows from Lemma~\ref{le:omega} that 
$\Theta \bar z \overline{k^{\Theta^2}_\lambda}\in K_\Theta\oplus \bar z \overline{K_\Theta}$
for $\lambda\in  \DDD \cup E(\Theta)$.

Take $g,h\in K_\Theta$, and, moreover, let $g\in L^\infty$. Then
\[
\<A^\Theta_{\phi_\lambda} g, h\> =\<\phi_\lambda g, h\> =\int_\TTT 
\Theta \bar z\overline{k^{\Theta^2}_\lambda} g \bar h\,dm.
\]
But $\Theta\bar z \bar h=\omega(h)\in K_\Theta$, $g\in K_\Theta\cap L^\infty$, and so 
by Lemma~\ref{le:product} $g\Theta \bar z \bar h\in K_{\Theta^2}$. Therefore
\[
\begin{split}
\int_\TTT \Theta \bar z\overline{k^{\Theta^2}_\lambda} g \bar h\,dm&=
\< g \omega(h), k^{\Theta^2}_\lambda\> = g(\lambda) (\omega(h))(\lambda) 
=\<g, k^\Theta_\lambda\> \<\omega(h), k^\Theta_\lambda\>\\
&=\<g, k^\Theta_\lambda\> \overline{\<h, \omega(k^\Theta_\lambda)\>}
= \<g, k^\Theta_\lambda\> \overline{\<h, \tilde k^\Theta_\lambda\>} =
\< (\tilde k^\Theta_\lambda\otimes {k}^\Theta_\lambda) g, h\>.
\end{split}
\]
Therefore $A_{\varphi_\lambda}^\Theta=\tilde k_\lambda^\Theta\otimes k_\lambda^\Theta$ 
as claimed. Finally, recall that, for $\zeta\in E(\Theta)$, we have 
$\tilde k_\zeta^\Theta=\omega (k^\Theta_\zeta)=\Theta(\zeta)\overline{\zeta}\,k^\Theta_\zeta$. 
\end{proof}

The construction of bounded truncated Toeplitz operators that 
have no bounded symbol is based on the next lemma.

\begin{lem}\label{le:estimation_phi_p}
Let $\Th$ be an inner function and $1<p<\infty$. There exists a constant $C$ depending only on $\Th$ and $p$ such that, if $\phi,\psi\in L^2$ are two symbols for the same truncated Toeplitz operator, with $\phi\in K_\Theta\oplus \bar z \overline{K_\Theta}$, then
\[
\|\phi\|_p\le C(\|\psi\|_p+\|\phi\|_2).
\] 
In particular, if $\psi\in L^p$, then $\phi\in L^p$.
\end{lem}

\begin{proof}
By hypothesis $P_{\mathfrak{S}_\Th}\phi= P_{\mathfrak{S}_\Th}\psi$; therefore, using~\eqref{eq:P_S},
\[
\phi= Q_\Theta\phi= P_{\mathfrak{S}_\Th}\phi+ \<\phi,q_\Th\>q_\Th =
P_{\mathfrak{S}_\Th}\psi+ \<\phi,q_\Th\>q_\Th.
\]
By Lemma~\ref{le:inclusion1} we have $\|P_{\mathfrak{S}_\Th}\psi\|_p\le 
 C_1\|\psi\|_p$, while
 \[
 \| \<\phi,q_\Th\>q_\Th\|_p\le \|\phi\|_2\cdot \|q_\Th\|_p,
 \]
whence the lemma follows.
\end{proof}

If $\Theta$ is an inner function and $\zeta\in E(\Th)$, then, as noted above,  $k^\Th_\zeta\otimes k^\Th_\zeta$ is a rank one operator in $\TT(K_\Th)$. In~\cite{Sarason} Sarason has asked specifically whether this operator has a bounded symbol. We can now show that in general this question has a negative answer.

\begin{thm}\label{th:counterex}
Suppose that $\Theta$ is an inner function which has an 
angular derivative at $\zeta\in \TTT$. Let $p\in (2,+\infty)$. Then the following are equivalent:
\begin{enumerate}
\item the bounded truncated Toeplitz operator $ k^\Theta_\zeta\otimes k^\Theta_\zeta$ has a symbol $\psi\in L^p$;
\item $k_\zeta^\Theta \in L^p$.
\end{enumerate}
In particular, if $k^\Theta_\zeta\not\in L^p$ for some $p \in (2,\infty)$, 
then $ k^\Theta_\zeta\otimes k^\Theta_\zeta$ is a bounded truncated
Toeplitz operator with no bounded symbol.
\end{thm}

\begin{proof} A symbol for the operator $k_\zeta^\Theta\otimes k_\zeta^\Theta$ is,
by Lemma~\ref{le:anton1}, $\phi=\overline{\Theta(\zeta)}\zeta \Theta \bar z \overline{k^{\Theta^2}_\zeta}$. Since by Lemma~\ref{Lem:comparaison-Theta-Theta2} $\phi\in L^p$ if and only if $k_\zeta^\Theta\in L^p$, we obtain that $(2)$ implies $(1)$.  Conversely, assume that $\psi\in L^p$ is a symbol for $ k^\Theta_\zeta\otimes k^\Theta_\zeta$. We may then apply Lemma~\ref{le:estimation_phi_p} and obtain that $ \phi\in L^p$. Once again according to Lemma~\ref{Lem:comparaison-Theta-Theta2}, we get that $k_\zeta^\Theta\in L^p$, which proves that $(1)$ implies $(2)$.
\end{proof}

To obtain a bounded truncated Toeplitz operator with no bounded symbol, it is sufficient to have a point $\zeta\in\TTT$ such that~\eqref{eq:cohn} is true for $p=2$ but 
not for some strictly larger value of $p$. It is now easy to give concrete examples, as, for instance:
\begin{enumerate}
 \item a Blaschke product with zeros $a_k$ accumulating to the point 1, 
and such that
\[
 \sum_{k}\frac{1-|a_k|^2}{|1-a_k|^2}<+\infty,\quad 
\sum_{k}\frac{1-|a_k|^2}{|1-a_k|^p}=+\infty\ \ \mbox{ for some } p>2;
\]
\item a singular function  $\sigma=\sum_k c_k \delta_{\zeta_k}$ 
with $\sum_k c_k< +\infty$, $\zeta_k\to 1$, and
\[
  \sum_{k}\frac{c_k}{|1-\zeta_k|^2}< +\infty,\quad 
  \sum_{k}\frac{c_k}{|1-\zeta_k|^p}= +\infty\ \ \mbox{ for some } p>2.
\]
\end{enumerate}

\begin{rem}\label{re:relate_question}
A related question raised in~\cite{Sarason} remains open. 
Let $\mu$ be a positive measure on $\TTT$ such that the support of the singular part of $\mu$ (with respect to the Lebesgue measure) is contained in $\TTT\setminus\sigma(\Theta)$, where $\sigma(\Theta)$ is the spectrum of the inner function $\Theta$. Then we say that $\mu$ is a {\em Carleson measure} for $K_\Theta$ if there is a constant $c>0$ such that
\begin{equation}\label{eq:carleson}
\int_\TTT |f|^2\,d\mu\leq c\|f\|_2^2,\qquad f\in K_\Theta.
\end{equation}
It is easy to see (and had already been noticed in~\cite{Cohn-JOT86}) that~\eqref{eq:carleson} is equivalent to the boundedness of the operator $A^\Theta_\mu$ defined by the formula
\begin{equation}\label{eq:definition-AmuTheta}
 \<\Ath_\mu f, g\>=\int_\TTT f\bar g\, d\mu,\qquad f,g\in K_\Theta;
\end{equation}
it is shown in~\cite{Sarason} that $\Ath_\mu$ is 
a truncated Toeplitz operator. More generally, a complex measure $\nu$ on $\TTT$ is called a Carleson measure for $K_\Theta$  if its total variation $|\nu|$ is a Carleson measure for $K_\Theta$. In this case there is a corresponding operator $A_\nu^\Theta$, defined also by formula \eqref{eq:definition-AmuTheta}, which belongs to $\TT(K_\Theta)$. Now if a truncated Toeplitz operator $\Ath_\phi$ has a bounded symbol $\psi\in L^\infty$ then the measure $d\mu=\psi\,dm$ is a Carleson measure for $K_\Theta$ and $\Ath_\phi=\Ath_\mu$. The natural question  
whether every operator in $\TT(\Kth)$ is of the form $\Ath_\mu$ (for some Carleson measure $\mu$ for $K_\Theta$) is not answered by our counterexample; indeed (as  
already noticed in~\cite{Sarason}) if $\Theta$ has 
an angular derivative in the sense of Carath\'eodory 
at $\zeta\in\TTT$, then $\delta_\zeta$ is a Carleson measure 
for $K_\Theta$ and $k^\Theta_\zeta\otimes k^\Theta_\zeta=A_{\delta_\zeta}^\Theta$.
\end{rem}

\begin{rem}\label{re:sarason_approach}
We arrive at the same class of counterexamples
as in Theorem \ref{th:counterex}
if we follow an idea due to Sarason \cite{Sarason}
(we would like to emphasize that our first counterexample was 
obtained in this way). It is shown in 
\cite[Section 5]{Sarason} that, for an inner function $\Theta$ which has 
an angular derivative at the point 
$\zeta \in \TTT$, the rank one operator 
$k^\Theta_\zeta\otimes k^\Theta_\zeta$ 
has a bounded symbol if and only if
there exists a function $h\in H^2$ such that
\begin{equation}\label{eq:sarason1}
\Rea \bigg( \frac{\overline{\Theta(\zeta)}\, \Theta}
{1-\overline \zeta z} +\Theta h\bigg) \in L^\infty.
\end{equation}

Since $\Re(1-\bar\zeta z)^{-1}=1/2$ a.e. on $\TTT$, condition (\ref{eq:sarason1}) 
is, obviously, equivalent to 
\[
 \Rea \big( k^\Theta_\zeta +\Theta h\big) \in L^\infty. 
\]
Then, by the M. Riesz theorem, 
$k^\Theta_\zeta +\Theta h \in L^p$ for any $p\in (2, \infty)$ and
the boundedness of the projection $P_\Theta$ in $L^p$
implies that $k^\Theta_\zeta \in L^p$. 
\end{rem}

The next theorem provides a wider class of examples.

\begin{thm}\label{th:generalcounter}
Suppose that $\Th$ is an inner function with the property that each bounded operator in  $\TT(K_\Th)$ has a bounded symbol. Then for each $p>2$ we have 
\begin{equation}\label{eq:supbdd}
\sup_{\lambda\in\DDD}\frac{\|k^\Th_\lambda\|_p}{\|k^\Th_\lambda\|_2^2}<\infty.
\end{equation}
\end{thm}

\begin{proof}
As mentioned in the previous section, it follows from the open mapping theorem 
that there exists a constant $C>0$ such that for any operator
$A \in \mathcal{T}(K_\Theta)$ one can always find  
a symbol $\psi \in L^\infty$ with $\|\psi\|_\infty \le C\|A\|$.

Fix $\lambda\in\DDD$, and consider the rank one operator $\tilde k^\Theta_\lambda\otimes {k}^\Theta_\lambda$, which has operator norm $\|k^\Theta_\lambda\|_2^2$. Therefore there exists $\psi_\lambda\in L^\infty$ with $\Ath_{\psi_\lambda}=\tilde k^\Theta_\lambda\otimes {k}^\Theta_\lambda$ and 
\begin{equation}\label{eq:est1}
\|\psi_\lambda\|_p\le \|\psi_\lambda\|_\infty\le C \|k^\Theta_\lambda\|_2^2.
\end{equation}

On the other hand, $\phi_\lambda=\Theta 
\bar z \overline{k^{\Theta^2}_\lambda}\in K_\Theta\oplus \bar z \overline{K_\Theta}$ is also a symbol for $\tilde k^\Theta_\lambda\otimes {k}^\Theta_\lambda$ by Lemma~\ref{le:anton1}. Applying Lemma~\ref{le:estimation_phi_p}, it follows that there exists a constant $C_1>0$ such that
\[
\|\phi_\lambda\|_p\le C_1(\|\psi_\lambda\|_p+\|\phi_\lambda\|_2).
\]
By~\eqref{eq:noyau-reproduisant-square} and Lemma~\ref{Lem:comparaison-Theta-Theta2}~(b), we have 
\begin{equation}\label{eq:est2}
\|\phi_\lambda\|_2=\|k^{\Th^2}_\lambda\|_2\le 2 \|k^{\Th}_\lambda\|_2\le C_2 \|k^{\Th}_\lambda\|_2^2.
\end{equation}
Therefore~\eqref{eq:est1} and~\eqref{eq:est2} yield
\begin{equation*}
\|\phi_\lambda\|_p\le C_1(C+C_2) \|k^{\Th}_\lambda\|_2^2.
\end{equation*}
Since $ \|\phi_\lambda\|_p= \|k^{\Th^2}_\lambda\|_p$, using once more~\eqref{eq:noyau-reproduisant-square} concludes the proof.
\end{proof}

It is easy to see that if there exists $\zeta\in E(\Th)$ such that $k^\Th_\zeta\not\in L^p$, then 
\[
\sup_{r<1}\frac{\|k^\Th_{r\zeta}\|_p}{\|k^\Th_{r\zeta}\|_2^2}=\infty.
\]
Therefore the existence of operators in $\TT(K_\Th)$ without bounded symbol, under the hypothesis of Theorem~\ref{th:counterex}, is also a consequence of Theorem~\ref{th:generalcounter}. Note however that Theorem~\ref{th:generalcounter} does not show that the particular operator $k_\zeta^\Theta\otimes k_\zeta^\Theta$ is a bounded truncated Toeplitz operator without bounded symbol. A larger class of examples is described below.

\begin{example}
Let $\Theta$ be a Blaschke product such that 
for some sequence of its zeros $z_n$ and some
points $\zeta_n \in \TTT$ (which are "close to $z_n$"), 
we have, for some $p \in (2,\infty)$,
\begin{equation}
\label{df1}
|\Theta'(\zeta_n)| = 
\|k^\Theta_{\zeta_{n}}\|_2^2 \asymp \frac{1-|z_n|}{|\zeta_n  -  z_n|^2}, 
\qquad 
\|k^\Theta_{\zeta_{n}}\|_p^p \asymp \frac{1-|z_n|}
{|\zeta_n-z_n|^p} 
\end{equation}
(notation $X\asymp Y$ means that the fraction $X/Y$ 
is bounded above and below by some positive constants), and 
\begin{equation}
\label{df3}
\lim_{n\to +\infty}
\frac{(1-|z_n|)^{1-\frac{1}{p}}}{|\zeta_n-z_n|} =0.
\end{equation}
Condition (\ref{df1}) means that
the main contribution to the norms of $k_{\zeta_n}^\Theta$
is due to the closest zero $z_n$. 
Then, by Theorem~\ref{th:generalcounter},
there exists a bounded truncated Toeplitz operator without bounded
symbol.

Such examples may be easily constructed. Take a sequence $w_k\in\DDD$
such that 
$w_k \to \zeta$ and
$$
\lim_{k\to +\infty} \frac{(1-|w_k|)^\gamma}{|w_k - \zeta|} = 0
$$
for some $\zeta \in \TTT$ and $\gamma \in (0,1)$.
Then it is not difficult to see that 
for any $p>\max (2, (1-\gamma)^{-1})$ one can construct recurrently 
a subsequence $z_n=w_{k_n}$ of $w_k$ and a sequence $\zeta_n\in\TTT$
with the properties (\ref{df1}) and (\ref{df3}).

Although related to the examples of Theorem \ref{th:counterex}, 
this class of examples may be different. Indeed, it is possible 
that $\Theta$ has no angular derivative at  
$\zeta$, e.g., if $1-|z_n| = |\zeta-z_n|^2$.
Also, if the zeros tend to $\zeta$ "very tangentially", 
it is possible that $k_\zeta^\Theta$ is in $L^p$ 
for any $p\in (2,\infty)$, but
there exists a bounded truncated Toeplitz operator without a bounded symbol. 
\end{example}

We pass now to the Reproducing Kernel Thesis. The next example shows 
that in general  Question 2 has a negative answer.

\begin{example}\label{ex:rkt}
Suppose $\Theta$ is a singular inner function and $s\in [0,1)$. Then
\[\begin{split}
A^\Theta_{\bar\Theta^s} {k}^\Theta_\lambda &=
P_\Theta\left(\frac{\bar\Theta^{s}-\overline{\Theta(\lambda)}\Theta^{1-s }}{1-\bar\lambda z} \right)\\
&=P_\Theta \left( \frac{\bar\Theta^s-\overline{\Theta(\lambda)^s} 
+\overline{\Theta(\lambda)^s} \big(1- \overline{\Theta(\lambda)^{1-s}} \Theta^{1-s }\big)
 }{1-\bar\lambda z} \right)\\
&=P_\Theta\left( \bar z  \frac{ \bar\Theta^s- \overline{\Theta^s(\lambda)}}{\bar z-\bar\lambda}  \right) + \overline{\Theta(\lambda)^s}P_\Theta\left( \frac{1- \overline{\Theta(\lambda)^{1-s}} \Theta^{1-s }}{1-\bar\lambda z} \right)\\
&=P_\Theta\left(\bar z \overline{\tilde k_\lambda^{\Theta^s}}\right) +\overline{\Theta(\lambda)^s}P_\Theta \left(k_\lambda^{\Theta^{1-s}}\right).
\end{split}
\]
The first term $\bar z \overline{\tilde k_\lambda^{\Theta^s}}$ is 
in $\bar z\overline{H^2}$, which is orthogonal to $K_\Theta$, 
while the second $k_\lambda^{\Theta^{1-s}}$ is contained in 
$ K_{\Theta^{1-s}}\subset K_\Theta$. Therefore we have
\[
A^\Theta_{\bar\Theta^s} {k}^\Theta_\lambda =
\overline{\Theta(\lambda)^s} k_\lambda^{\Theta^{1-s}},
\]
and
\[
\|A^\Theta_{\bar\Theta^s}{k}^\Theta_\lambda\|_2^2 = |\Theta(\lambda)|^{2s} \frac{1-|\Theta(\lambda)|^{2-2s}}{1-|\lambda|^2}, 
\qquad
\|A^\Theta_{\bar\Theta^s} {h}^\Theta_\lambda\|_2^2 =
\frac{|\Theta(\lambda)|^{2s} ( 1-|\Theta(\lambda)|^{2-2s} )}{ 1-|\Theta(\lambda)|^2}.
\]
It is easy to see that $\sup_{y\in[0,1)}\frac{y^s-y}{1-y}\le 1-s\to0$ when $s\to 1$, and therefore
\[
\rho_r(A^\Theta_{\bar\Theta^s})=   \sup_{\lambda\in\DDD} \|A^\Theta_{\bar\Theta^s} 
{h}^\Theta_\lambda\|_2^2\to 0 \quad \mbox{ for }\  s\to 1.
\]

On the other hand, $\Theta^s K_{\Theta^{1-s}}\subset K_\Theta$ and $\bar\Theta^s(\Theta^s K_{\Theta^{1-s}}) = K_{\Theta^{1-s}}\subset K_\Theta$; therefore $A^\Theta_{\bar\Theta^s}$ acts isometrically on $\Theta^s K_{\Theta^{1-s}}$, so it has norm 1. Thus there is no constant $M$ such that 
\[
\|A_\phi^\Theta\|\le M \sup_{\lambda\in\DDD} \rho_r(A^\Theta_{\phi}) \]
for all $\phi$. 
\end{example}

It seems natural to deduce that in the previous example we may actually have 
a truncated Toeplitz operator which is uniformly bounded on reproducing kernels 
but not bounded. This is indeed true, by an abstract argument based on 
Proposition~~\ref{pr:recapturing}. Note that the quantity $\rho_r$ introduced 
in \eqref{eq:rho} is a norm, and $\rho_r(T)\le \|T\|$, for every linear operator 
$T$ whose domain contains $H^\infty\cap K_\Theta$.

\begin{prop}\label{pr:rho}
Assume that for any (not necessarily bounded) truncated Toeplitz 
operator $A$ on $K_\Theta$ the inequality $\rho_r(A)<\infty$ implies 
that $A$ is bounded. Then $\TT(K_\Theta)$ is complete with respect 
to $\rho_r$, and $\rho_r$ is equivalent to the operator norm on $\TT(K_\Theta)$.
\end{prop}

\begin{proof} 
Fix $\mu\in\DDD$ such that $\Theta(\mu)\not=0$. Let $A^\Theta_{\phi_n}$ be a $\rho_r$-Cauchy sequence in $\TT(K_\Theta)$. Suppose all $\phi_n$ are 
written as $\phi_n=\phi_{n,+} + \overline{\phi_{n,-}}$, with $\phi_{n,+}, \phi_{n,-} \in K_\Theta$, and $\phi_{n,-}(\mu)=0$. 
According to~~\eqref{eq:phi-rho}, the sequences $\phi_{n,\pm}$ are Cauchy 
sequences in $K_\Theta$ and thus converge to functions $\phi_\pm\in K_\Theta$; moreover we also have $\phi_-(\mu)=0$ (because norm convergence in $H^2$ implies pointwise convergence). Define then $\phi=\phi_+ + \overline{\phi_-}\in L^2$. By~\eqref{eq:recap1}, we have
\[
A_{\phi_n}^\Theta k_\lambda^\Theta=\omega\big[(I-\lambda S^*)^{-1}\left(\omega(\phi_{n,+})+\phi_{n,-}(\lambda)S^*\Theta-\Theta(\lambda)S^*\phi_{n,-}\right)\big],
\] 
so the sequence $\Ath_{\phi_n}k_\lambda^\Theta$ tends (in $K_\Theta$) 
to $\Ath_\phi k_\lambda^\Theta$, for all $\lambda\in\DDD$. In particular, 
we have $\rho_r(\Ath_\phi)<+\infty$, whence $\Ath_\phi\in\TT(K_\Theta)$. 
Now it is easy to see that $A^\Theta_{\phi_n}\to A^\Theta_\phi$ in the $\rho_r$-norm.

Thus $\TT(K_\Theta)$ is indeed complete with respect to the $\rho_r$-norm. 
The equivalence of the norms is then a consequence of the open mapping theorem.
\end{proof}

Proposition~\ref{pr:rho} and Example~\ref{ex:rkt} imply that, if $\Theta$ is a singular inner 
function, then there exist truncated Toeplitz operators $\Ath_\phi$ with $\rho_r(\Ath_\phi)$ 
finite, but $\Ath_\phi$ unbounded. Therefore Question~2 has a negative answer for a rather large class of inner functions $\Theta$. 
If we consider such a truncated Toeplitz operator, then its adjoint, $\Ath_{\bar\phi}$, is an unbounded truncated Toeplitz operator with $\rho_d( \Ath_{\bar\phi})=\rho_r(A_\phi^\Theta)<+\infty$. 

It is easy to see, however, that in Example~\ref{ex:rkt} $\rho_d(A^\Theta_{\bar\Theta^s})=1$ for all $s<1$. This suggests that we should rather consider boundedness of the action of the operator on both the reproducing kernels and the difference quotients, and that the quantity $\rho$ might be a better estimate for the norm of a truncated Toeplitz operator than either $\rho_r$ or $\rho_d$. We have been thus lead to formulate Question~3 as a more relevant variant of the RKT; further arguments will appear in the next section.

\mysection{Positive results}\label{mysection:positive-results}

There are essentially two cases in which one can give  positive answers to Questions~1 and~3. There are similarities between them: in both one obtains a convenient decomposition of the symbol in three parts: one analytic, one coanalytic, and one that is neither analytic nor coanalytic, but well controlled. 

\subsection{A general result}
As we have seen in Proposition~\ref{pr:bonsall}  and \ref{pr:cobonsall}, 
the answers to Questions~1 and~3 are positive for classes of truncated Toeplitz operators corresponding to analytic and coanalytic symbols. We complete these propositions with a different boundedness result, which covers certain cases when the symbol is neither analytic nor coanalytic. The proof is based on an idea of Cohn~\cite{cohn1}. 
\begin{thm}\label{th:centralsymbol}
Suppose $\theta$ and $\Theta$ are two inner functions such that $\theta^3$ divides $z\Theta$ and $\Theta$ divides $\theta^4$. If $\phi\in K_\theta+\overline{K_\theta}$ then $\|\phi\|_\infty\le 2 \rho_r(A^\Theta_\phi)$. 
\end{thm}

\begin{proof} Using Lemma~\ref{le:technique-espace-modele}, if $f\in L^\infty\cap\theta K_\theta$, then $f\in K_\Theta$ and $\bar\phi f \in \Kth$; thus $A^\Theta_{\bar\phi} f=\bar\phi f$. 
If we write $f=\theta f_1$, $\phi_1=\theta\bar\phi$, then 
$\phi_1\in H^2$, $f_1\in K_\theta$, and $\phi_1 f_1=\bar\phi f = A^\Theta_{\bar\phi} f\in \Kth$. Therefore, for $\lambda\in\DDD$, 
\[
\begin{split}
|\phi_1(\lambda)f_1(\lambda)|&=|\langle \phi_1 f_1,k_\lambda^\Theta \rangle|=|\langle \theta f_1,\phi k_\lambda^\Theta \rangle|=|\langle \theta f_1,A_\phi^\Theta k_\lambda^\Theta\rangle| \\
&\le \|f_1\|\|A_\phi^\Theta k_\lambda^\Theta\|_2 \le \|f_1\| \|k_\lambda^\Theta\|_2\rho_r(A_\phi^\Theta),
\end{split}
\]
where we used the fact that $\theta f_1\in K_\Theta$. 

For a fixed $\lambda\in\DDD$, 
\[
\mathop{\sup_{f_1\in  K_\theta\cap L^\infty}}_{\|f_1\|_2\le 1} |f_1(\lambda)|= \mathop{\sup_{f_1\in  K_\theta\cap L^\infty}}_{\|f_1\|_2\le 1} |\<f_1, k^\theta_\lambda\>| = \|k^\theta_\lambda \|_2,
\]
and thus
\[
|\phi_1(\lambda)|\le \rho_r(A^\Theta_\phi) \frac{\|k^{\Theta}_\lambda\|_2}{\|k^\theta_\lambda \|_2}
= \rho_r(A^\Theta_\phi) \frac{(1-|\Theta(\lambda)|^2)^{1/2}}{(1-|\theta(\lambda)|^2)^{1/2}}.
\]
If $\Theta$ divides $\theta^4$, then $|\Theta(\lambda)|\ge |\theta(\lambda)|^4$, and therefore
\[
1-|\Theta(\lambda)|^2\le 1-|\theta(\lambda)|^8\le 4(1-|\theta(\lambda)|^2).
\]
It follows that $|\phi_1(\lambda)|\le 2\rho_r(A^\Theta_\phi)$ for all $\lambda\in \DDD$, and thus $\|\phi_1\|_\infty \le 2\rho_r(A^\Theta_\phi)$. The proof is finished by noting that $\|\phi\|_\infty=\|\phi_1\|_\infty$.
\end{proof}

%
%

As a consequence, we obtain a general result for the existence of bounded symbols and Reproducing Kernel Thesis.

\begin{cor}\label{co:lcr}
Let $\Theta$ be an inner function and assume that there is another inner function $\theta$ such that $\theta^3$ divides $z\Theta$ and $\Theta$ divides $\theta^4$. Suppose also there are constants $C_i>0$, $i=1,2,3$ such that any $\phi\in L^2$ can be written as $\phi=\phi_1+\phi_2+\phi_3$, with:

(a) $\phi_1\in  K_\theta+\overline{K_\theta}$, $\phi_2\in H^2$, and $\phi_3\in \overline{H^2}$;

(b) $ \rho(A_{\phi_i}^\Theta)\le C_i \rho(A_\phi^\Theta)$ for $i=1,2,3$.

Then the following are equivalent:
\begin{enumerate}
\item[$\mathrm{(i)}$] $A_\phi^\Theta$ has a bounded symbol;
\item[$\mathrm{(ii)}$]  $A_\phi^\Theta$ is bounded;
\item[$\mathrm{(iii)}$]  $\rho(A_\phi^\Theta)<+\infty$.
\end{enumerate}
More precisely, there exists a constant $C>0$ such that any truncated Toeplitz operator $A_\phi^\Theta$ has a symbol $\phi_0$ with $\|\phi_0\|_\infty \le C \rho(A_\phi^\Theta)$.

\end{cor}

There are of course many decompositions of $\phi$ as in (a);  the difficulty  consists in finding one that satisfies (b). 

\begin{proof}
It is immediate that $\mathrm{(i)}\Longrightarrow\mathrm{(ii)}\Longrightarrow \mathrm{(iii)}$, so it remains to prove $\mathrm{(iii)}\Longrightarrow\mathrm{(i)}$. Since $\rho(A_{\phi_i}^\Theta)<+\infty$, $i=2,3$, Proposition~\ref{pr:bonsall}  and \ref{pr:cobonsall} imply that $A_{\phi_i}^\Theta$ have  bounded symbols $\tilde\phi_i$ with $\|\tilde\phi_i\|_\infty\le \tilde C \rho(A_{\phi_i}^\Theta)\le \tilde C C_i \rho(A_\phi^\Theta)$. As for $\phi_1$, we can apply Theorem~\ref{th:centralsymbol} which gives that $\phi_1$ is bounded with $\|\phi_1\|_\infty\le 2 \rho_r(A_{\phi_1}^\Theta)\le 2 C_1 \rho(A_\phi^\Theta)$. Finally $A_\phi^\Theta$ has the bounded symbol $\phi_0=\phi_1+\tilde\phi_2+\tilde\phi_3$ whose norm is at most $(2C_1+\tilde C(C_2+C_3))\rho(A_\phi^\Theta)$.
\end{proof}

\subsection{Classical Toeplitz matrices}

Suppose $\Theta(z)=z^N$; the space $K_\Theta$ is then an $N$-dimensional space with orthonormal basis formed by monomials, and truncated Toeplitz operators have a (usual) Toeplitz matrix with  respect of this basis. Of course every truncated Toeplitz operator has a bounded symbol; it is however interesting that there exists a universal estimate of this bound. The question had been raised in~\cite[Section 7]{Sarason}; the positive answer had actually been already independently obtained in~\cite{BT01} and~\cite{NF03}. The following result is stronger, giving a universal estimate for the symbols in terms of the action on the reproducing kernels. 

\begin{thm}\label{th:toeplitz} Suppose $\Theta(z)=z^N$. There exists a constant $C>0$, independent of $N$, such that
any  truncated Toeplitz operator  $\Ath_\phi$ has a symbol $\phi_0\in L^\infty$ such that  $\|\phi_0\|_\infty\le C  \rho(A_\phi^\Theta ) $.
\end{thm}
\begin{proof} 
Consider a smooth function $\eta_k$ on $\TTT$, and the convolution (on $\TTT)$ 
$\phi_k=\eta_k*\phi$, that is,
\[
\phi_k(e^{is})=\frac{1}{2\pi} \int_{-\pi}^\pi \eta_k(e^{it})\phi(e^{i(s-t)})\, dt.
\]
We have then $\hat \phi_k(n)=\hat \eta_k(n)\hat \phi(n)$, $n\in\ZZZ$.

The map $\tau_t$ defined by $\tau_t:f(z)\mapsto f(e^{it}z)$ is a unitary on $K_{\Theta}$ and straightforward computations show that 
\begin{equation}\label{eq:action-tau-t-noyau}
\tau_t h_\lambda^{\Theta}=h_{e^{-it}\lambda}^{\Theta}\quad\hbox{and}\quad  \tau_t \tilde h_\lambda^{\Theta}= e^{i(N-1)t} \tilde h_{e^{-it}\lambda}^{\Theta},
\end{equation}
for every $\lambda\in\DDD$. By Fubini's Theorem and a change of variables we have
\[
\<A^{\Theta}_{\phi_k} f,g\>=\frac{1}{2\pi} \int_{-\pi}^\pi \eta_k(e^{it})  \<A^{\Theta}_\phi \tau_t(f),\tau_t(g)\>\,dt,
\]
for every $f,g\in K_{\Theta}$. That implies that 
\[
\|A^{\Theta}_{\phi_k} h_{\lambda}^{\Theta}\|=\sup_{\substack{g\in K_{\Theta}\\ \|g\|_2\leq 1}} \left|\langle  A^{\Theta}_{\phi_k} h_{\lambda}^{\Theta},g\rangle\right|\leq \sup_{\substack{g\in K_{\Theta}\\ \|g\|_2\leq 1}} \frac{1}{2\pi}\int_{-\pi}^\pi |\eta_k(e^{it})| |\<A^{\Theta}_\phi \tau_t(h_\lambda^{\Theta}),\tau_t(g)\>|\,dt,
\]
and using \eqref{eq:action-tau-t-noyau}, we obtain
\[
\|A_{\phi_k}^{\Theta} h_\lambda^{\Theta}\|\leq \|\eta_k\|_1 \rho_r(A_{\phi}^{\Theta})\leq \|\eta_k\|_1 \rho(A_{\phi}^{\Theta}).
\]
A similar argument shows that 
\[
\|A_{\phi_k}^{\Theta} \tilde h_\lambda^{\Theta}\|\leq \|\eta_k\|_1 \rho(A_{\phi}^{\Theta})
\]
and thus
\begin{equation}\label{eq:inegalite-cle-rho'}
\rho(A_{\phi_k}^{\Theta})\leq \|\eta_k\|_1 \rho(A_\phi^\Theta).
\end{equation}

Now consider the Fej\'er kernel $F_m$, defined by the formula $\hat F_m(n)=1-\frac{|n|}{m}$ for $|n|\le m$ and $\hat F_m(n)=0$ otherwise. It is well known that $\|F_m\|_1=1$ for all $m\in \NNN$. If we take $M=\left[\frac{N+1}{3}  \right]$ and define $\eta_i$ ($i=1,2,3$) by
\[
\eta_1=  F_M, \qquad \eta_2=2e^{2iMt}F_{2M}- e^{2iMt}F_{M}, \qquad \eta_3=\bar \eta_2,
\]
then $\hat \eta_2(n)=0$ for $n< 0$, $\hat \eta_3(n)=0$ for $n>0$,  $\hat \eta_1(n)+\hat \eta_2(n)+\hat \eta_3(n)=1$ for $|n|\le N$, and $\|\eta_1\|_1=1$, $\|\eta_i\|_1\le 3$ for $i=2,3$. If we denote $\phi_i=\eta_i*\phi$, then $\phi=\phi_1+\phi_2+\phi_3$, $\phi_1\in K_{z^M}+\overline{K_{z^M}}$, $\phi_2$ is analytic and $\phi_3$ is coanalytic. Moreover $z^{3M}$ divides $z^{N+1}$ and $z^N$ divides $z^{4M}$. According to \eqref{eq:inegalite-cle-rho'}, we can apply Corollary~\ref{co:lcr} to obtain that there exists a universal constant $C>0$ such that $A_\phi^\Theta$ has a bounded symbol $\phi_0$ with $\|\phi_0\|_\infty\le C \rho(A_\phi^\Theta)$.

\end{proof}

In particular, it follows from Theorem~\ref{th:toeplitz} that any (classical) Toeplitz matrix $A^{z^N}_{\phi}$  has a symbol $\phi_0$ such that  $\|\phi_0\|_\infty\le C \|A^{z^N}_\phi\|$. The similar statement is proved with explicit estimates  $\|\phi_0\|_\infty\le 4 \|A^{z^N}_\phi\|$ in~\cite{BT01} and $\|\phi_0\|_\infty\le 3 \|A^{z^N}_\phi\|$ in~\cite{NF03}.

We can obtain a slightly more general result (in the choice of the function $\Theta$). 

\begin{cor}Suppose $\Theta=b_\alpha^N$, with $b_\alpha(z)=\frac{\alpha-z}{1-\bar\alpha z}$ a Blaschke factor. There exists a universal constant $C>0$ such that
any  truncated Toeplitz operator  $\Ath_\phi$ has a symbol $\phi_0\in L^\infty$ such that  $\|\phi_0\|_\infty\le C  \rho(A_\phi^\Theta ) $.
\end{cor}

\begin{proof}

The mapping $U$ defined by 
\[
(U(f))(z):=\frac{(1-|\alpha|^2)^{1/2}}{1-\bar\alpha z} f(b_\alpha(z)),\qquad 
z\in\DDD,\,f\in H^2,
\]
is unitary on $H^2$ and one easily checks that $UP_{z^N}=P_\Theta U$. In particular, it implies that $U(K_{z^N})=K_\Theta$;  straightforward computations show that 
\begin{equation}\label{eq:lien-noyau-reproduisant}
U h_\lambda^{z^N}=c_\lambda h^\Theta_{b_\alpha(\lambda)}\quad\hbox{and}\quad U \tilde h_\lambda^{z^N}=-\bar c_\lambda \tilde h^\Theta_{b_\alpha(\lambda)},
\end{equation}
for every $\lambda\in\DDD$, where $c_{\lambda}:=|1-\bar\lambda\alpha|(1-\bar\lambda\alpha)^{-1}$ is a constant of modulus one. 

Suppose $\Ath_\phi$ is a (bounded) truncated Toeplitz operator;  if $\Phi=\phi\circ b_\alpha$, then the relation $UP_{z^N}=P_\Theta U$ yields  $A^{z^N}_{\Phi }=U^*\Ath_\phi U$. Thus, using \eqref{eq:lien-noyau-reproduisant}, we obtain 
\[
\|A^{z^N}_{\Phi} h_\lambda^{z^N}\|_2=\|U^*\Ath_\phi U h_\lambda^{z^N}\|_2=\|\Ath_\phi h_{b_\alpha(\lambda)}^\Theta\|_2
\]
and 
\[
\|A^{z^N}_{\Phi} \tilde h_\lambda^{z^N}\|_2=\|U^*\Ath_\phi U \tilde h_\lambda^{z^N}\|_2=\|\Ath_\phi \tilde h_{b_\alpha(\lambda)}^\Theta\|_2,
\]
which implies that 
\begin{equation}\label{eq:egalite-rho'}
\rho(A^{z^N}_{\Phi })=\rho(\Ath_\phi )\,.
\end{equation}
Now it remains to apply Theorem~\ref{th:toeplitz} to complete the proof.
\end{proof}

\subsection{Elementary singular inner functions} 
Let us now take $\Theta(z)=\exp(\frac{z+1}{z-1})$. A positive answer to   Questions 1 and 3  is a consequence of results obtained by Rochberg~\cite{R87} and Smith~\cite{martin} on the Paley--Wiener space. We sketch the proof for completeness, without entering into details.

\begin{thm}\label{th:singular}
If $\Theta(z)=\exp(\frac{z+1}{z-1})$ and $\Ath_\phi$ is a truncated Toeplitz operator, 
then the following are equivalent:

{\rm(i)} $\Ath_\phi$ has a bounded symbol;

{\rm(ii)} $\Ath_\phi$ is bounded;

{\rm(iii)} $\rho(\Ath_\phi)<\infty$.

More precisely,
there exists a constant $C>0$ such that any truncated Toeplitz 
operator $\Ath_\phi$ has a symbol $\phi_0$ with 
$\|\phi_0\|_\infty\le C\rho(\Ath_\phi)$.
\end{thm}

\begin{proof} By Remark~\ref{re:halfspace} it is enough 
to prove the corresponding result for the space $\bm K_{\bm \Theta}$, 
where $\bm\Theta(w)=e^{iw}$, and  $\bm\rho$ is the analogue of $\rho$ 
for operators on $\bm K_{\bm \Theta}$. If $\FF$ denotes 
the Fourier transform on $\RRR$, then  $\bm K_{\bm \Theta}=\FF^{-1}(L^2([0,1]))$, 
and  we may suppose that the symbol $\bm\phi\in (t+i) \FF^{-1}(L^2([-1,1]))$.

For a rapidly decreasing function $\eta$ on $\RRR$, define
\begin{equation}\label{eq:eta}
\Psi(s)=\int_\RRR \eta(t) \bm\phi (s-t)\, dt.
\end{equation}
We have then $\hat\Psi=\hat\eta \hat {\bm\phi}$ and 
$\bm\rho({\bm A}^{\bm \Theta}_\psi)\le 
\|\eta\|_1\cdot \bm\rho({\bm A}^{\bm \Theta}_\phi)$.

Take now $\psi_i$, $i=1,2,3$, such that $\supp \hat\psi_1\subset [-1/3, 1/3]$, 
$\supp \hat\psi_2\subset [0, 2]$, $\supp \hat\psi_3\subset [-2, 0]$, 
and $\hat\psi_1+\hat\psi_2+\hat\psi_3=1$ on $[-1, 1]$. If 
we define $\bm\phi_i$ by replacing $\eta $ with $\psi_i$ 
in~\eqref{eq:eta}, then there is a constant $C_1>0$ such that 
$\bm\rho({\bm A}^{\bm \Theta}_{\bm\phi_i})
\le C_1 \bm\rho({\bm A}^{\bm \Theta}_{\bm\phi})$ for $i=1,2,3$.

On the other hand, $\bm\phi=\bm\phi_1+\bm\phi_2+\bm\phi_3$, 
$\bm\phi_1\in \bm{K}_{\bm{\Theta}^{1/3}}+\overline{\bm{K}_{\bm{\Theta}^{1/3}}}$,
 $\bm\phi_2$ is analytic, $\bm\phi_3$ is antianalytic. 
We may then apply the analogue of Corollary~\ref{co:lcr} 
for the upper half-plane which completes the proof. 

\end{proof}

One can see easily that a similar result is valid for any elementary 
singular function $\Theta(z)=\exp\left( a \frac{z+\zeta}{z-\zeta} \right)$, 
for $\zeta\in\TTT$, $a>0$.

\begin{rem}
Truncated Toeplitz operators on the model space 
$\bm K_{\bm \Theta}$ with $\bm\Theta(w)=e^{iaw}$ 
are closely connected with the so-called 
truncated Wiener--Hopf operators. Let  $\bm\phi \in L^1(\mathbb{R})$
and let 
$$
(W_{\bm\phi} f)(x) = \int_0^a f(t)\bm \phi(x-t)dt, \quad x\in (0,a),
$$
for $f\in L^2(0,a) \cap L^\infty(0,a)$. If $W$ extends to a
bounded operator on $L^2(0,a)$, then it is called a 
{\it truncated Wiener--Hopf operator}. If $\bm\phi =\hat{\bm\psi}$ 
with $\bm\psi \in (t+i)L^2(\mathbb{R})$ (the Fourier transform 
may be understood in the distributional sense), then
\[
W_{\bm\phi}f = \FF P_{\bm\Theta} (\bm\psi g)
\]
for $g = \check{f} \in \bm K_{\bm \Theta}$. 
Thus, the Wiener--Hopf operator $W_{\bm\phi}$
is unitarily equivalent to ${\bm A}^{\bm \Theta}_{\bm\psi}$.
\end{rem}

\mysection{Truncated Toeplitz operators with positive symbols}
As noted in Remark~\ref{re:relate_question}, if $\phi\in L^2$ 
is a positive function, then $\Ath_\phi$ is bounded if and only if 
$\phi\,dm$ is a Carleson measure for $K_\Theta$.
As a consequence mainly of results of Cohn~\cite{CohnPac82, Cohn-JOT86}, 
one can say more for positive symbols $\phi$ for a special class of model 
spaces. Recall that $\Theta$ is said to satisfy the {\emph {connected 
level set condition}} (and we write $\Theta\in (CLS)$) if there is 
$\varepsilon\in (0,1)$ such that the level set 
\[
\Omega(\Theta,\varepsilon):=\{z\in\DDD:|\Theta(z)|<\varepsilon\}
\]
is connected. Such inner functions are also referred to as 
{\emph {one-component}} inner functions.  

\begin{thm}\label{th:positive}
 Let $\Theta$ be an inner function such that $\Theta\in (CLS)$.  
If $\phi$ is a positive function in $L^2$, then the following conditions are equivalent:
\begin{enumerate}
\item $A_\phi^\Theta$ is a bounded operator on $K_\Theta^2$;
\item $\sup_{\lambda\in\DDD}\|A_\phi^\Theta h_\lambda^\Theta\|_2<+\infty$;
\item $\sup_{\lambda\in\DDD}|\langle 
A_\phi^\Theta h_\lambda^\Theta,h_\lambda^\Theta \rangle|<+\infty$;
\item $A_\phi^\Theta$ has a bounded symbol.
\end{enumerate}
\end{thm}

\begin{proof} The implications $(4)\Longrightarrow(1)\Longrightarrow(2)\Longrightarrow(3)$ 
are obvious.

We have
\begin{equation}\label{eq:cohn7}
\int_\TTT \phi|h_\lambda^\Theta|^2\,dm=
\langle \phi h_\lambda^\Theta,h_\lambda^\Theta \rangle=
\langle P_\Theta\phi h_\lambda^\Theta,h_\lambda^\Theta \rangle
=\langle A_\phi^\Theta h_\lambda^\Theta,h_\lambda^\Theta\rangle.
\end{equation}
It is shown in~\cite{CohnPac82} that, for $\Theta\in (CLS)$, a positive $\mu$ satisfies $\sup_{\lambda\in\DDD} \|h^\Theta_\lambda\|_{L^2(\mu)}<\infty$ if and only if it  is a Carleson measure for $\Kth$. Thus (3) implies that $\phi\,dm$ is a Carleson measure for $\Kth$, which has been noted above to be equivalent to $\Ath_\phi$ bounded; so $(1)\Longleftrightarrow(3)$.

On the other hand, it is proved in \cite{Cohn-JOT86} that if $\Ath_\phi$ is bounded, then
there are functions $v\in L^\infty(\TTT)$ and $h\in H^2$ such that $\phi=\Re(v+\Theta h)$.
Write then
\[
\phi=\Re v+\frac{1}{2}(\Theta h+\bar\Theta\bar h),
\]
which implies that $\phi-\Re v\in \Theta H^2+\overline{\Theta H^2}$. Therefore $A_\phi^\Theta=A_{\Re v}^\Theta$ and $\Re v\in L^\infty(\TTT)$. Thus the last remaining implication $(1)\Longrightarrow(4)$ is proved.
\end{proof}

\begin{rem}
In~\cite{CohnPac82}, Cohn asked the following question: let $\Theta$ be an inner function and let $\mu$ be a positive measure on $\TTT$ such that the singular part of $\mu$ is supported on a subset of $\TTT\setminus\sigma(\Theta)$; is it sufficient to have 
\[
\sup_{\lambda\in\DDD}\int_\TTT|h_\lambda^\Theta|^2\,d\mu<+\infty,
\]
to deduce that $\mu$ is a Carleson measure for $K_\Theta$? In~\cite{NV02} Nazarov 
and Volberg construct a counterexample to this question with a measure $\mu$ of the 
form $d\mu=\phi\,dm$ where $\phi$ is some positive function in $L^2$. In our context, 
this means that they provide an inner function $\Theta$ and a positive function 
$\phi\in L^2$ such that 
\begin{equation}\label{eq:condition-faible-RKT}
\sup_{\lambda\in\DDD}|\<\Ath_\phi h_\lambda^\Theta,h_\lambda^\Theta\>|<+\infty,
\end{equation}
while $\Ath_\phi$ is not bounded. But the condition 
\eqref{eq:condition-faible-RKT} is obviously weaker 
than $\rho_r(\Ath_\phi)<+\infty$ (note that since $\phi$ is positive, 
the truncated Toeplitz operator is positive and $\rho_r(\Ath_\phi)=\rho(\Ath_\phi)$). 
Thus an answer to Question 3 does not follow from the Nazarov--Volberg result.
\end{rem}

\begin{rem}
It is shown by Aleksandrov \cite[Theorem 1.2]{Aleksandrov} 
that the condition 
\[
\sup_{\lambda \in \DDD}\frac{\|k_\lambda^\Theta\|_\infty}{\|k_\lambda^\Theta\|_2^2}< +\infty
\]
is equivalent to $\Theta\in (CLS)$.
On the other hand, as we have seen in Theorem~\ref{th:generalcounter}, the condition 
\[
\sup_{\lambda\in \DDD} 
\frac{\|k_\lambda^\Theta\|_p}{\|k_\lambda^\Theta\|_2^2} 
= +\infty
\]
for some $p\in (2,\infty)$ implies that
there exists a bounded  operator in $\TT(K_\Th)$ without a bounded symbol. 
Therefore, based on Theorem~\ref{th:positive} and  
Theorem~\ref{th:generalcounter}, it seems reasonable to state the following conjecture.

\begin{conjecture}
Let $\Theta$ be an inner function. Then any bounded truncated Toeplitz
operator has a bounded symbol if and only if $\Theta\in (CLS)$.
\end{conjecture}

\end{rem}

\def\cprime{$'$}

\end{document}